\documentclass[12pt,multicol,english]{amsart}
\usepackage{multicol}
\usepackage[T1]{fontenc}
\usepackage[utf8]{inputenc}
\usepackage[a4paper]{geometry}
\usepackage{amssymb,amsmath}
\usepackage{version}
\usepackage{mathtools} 
\usepackage{enumerate}
\usepackage{dsfont}
\usepackage{mathdots}
\makeatletter

\usepackage{latexsym}
\usepackage{color}

\def\subsubsection{\@startsection{subsubsection}{3}%
	\z@{.5\linespacing\@plus.7\linespacing}{-.5em}%
	{\normalfont\bfseries}}
\newtheorem{thm}{Theorem}[section]
\newtheorem{lemme}[thm]{Lemma}

\newtheorem{prop}[thm]{Proposition}

\newtheorem{coro}[thm]{Corollary}
\newtheorem{defi}[thm]{Definition}
\newtheorem{rem}[thm]{Remark}
\newtheorem{rems}[thm]{Remarks}
\newtheorem{quest}[thm]{Question}

\newtheorem*{thm*}{Theorem}

\def\D{{\mathbb{D}}}

\def\R{{\mathbb{R}}}
\def\C{{\mathbb{C}}}

\def\N{{\mathbb{N}}}
\def\Z{{\mathbb{Z}}}
\def\T{{\mathbb{T}}}
\def\UU{{\mathcal{U}}}
\def\TT{{\mathcal{T}}}

\def\FF{{\mathcal{F}}}

\def\CC{{\mathcal{C}}}

\def\VV{{\mathcal{V}}}

\def\e{\varepsilon}
\def\dlow{\underline{d}}

\newcommand{\vp}{\varphi}

\numberwithin{equation}{section} 
\numberwithin{figure}{section} 
\numberwithin{table}{section} 

\usepackage[bookmarksopen, naturalnames]{hyperref}

\makeatother

\usepackage{babel}

\begin{document}
	\title{Abel universal functions}
	\author{S. Charpentier, A. Mouze}
	
	\address{St\'ephane Charpentier, Institut de Math\'ematiques, UMR 7373, Aix-Marseille
		Universite, 39 rue F. Joliot Curie, 13453 Marseille Cedex 13, France}
	\email{stephane.charpentier.1@univ-amu.fr}
	
	\address{Augustin Mouze, Laboratoire Paul Painlev\'e, UMR 8524, Cit\'e Scientifique, 59650 Villeneuve
		d'Ascq, France, Current address: \'Ecole Centrale de Lille, Cit\'e Scientifique, CS20048, 59651
		Villeneuve d'Ascq cedex, France}
	\email{augustin.mouze@univ-lille.fr}
	
	\thanks{The authors are supported by the grant ANR-17-CE40-0021 of the French National Research Agency ANR (project Front).}
	
	\begin{abstract} Given a sequence $\varrho=(r_n)_n\in [0,1)$ tending to $1$, we consider the set $\UU_A(\D,\varrho)$ of Abel universal functions consisting of holomorphic functions $f$ in the open unit disc $\mathbb{D}$ such that for any compact set $K$ included in the unit circle $\T$, different from $\T$, the set $\{z\mapsto f(r_n \cdot)\vert_K:n\in\mathbb{N}\}$ is dense in the space $\CC(K)$ of continuous functions on $K$. It is known that the set $\UU_A(\D,\varrho)$ is residual in $H(\mathbb{D})$. We prove that it does not coincide with any other classical sets of universal holomorphic functions. In particular, it is not even comparable in terms of inclusion to the set of holomorphic functions whose Taylor polynomials at $0$ are dense in $\CC(K)$ for any compact set $K\subset \T$ different from $\T$. Moreover we prove that the class of Abel universal functions is not invariant under the action of the differentiation operator. Finally an Abel universal function can be viewed as a universal vector of the sequence of dilation operators $T_n:f\mapsto f(r_n \cdot)$ acting on $H(\mathbb{D})$. Thus we study the dynamical properties of $(T_n)_n$ such as the multi-universality and the (common) frequent universality. All the proofs are constructive.
	\end{abstract}
	
	\maketitle
	
	\section{Introduction} Let $\mathbb{D}:=\{z\in\mathbb{C}:\vert z\vert<1\}$ denote the open unit disc, $\T=\{z\in \C:\, \vert z\vert=1\}$ the unit circle and let $H(\mathbb{D})$ be the Fréchet space consisting of all holomorphic functions on $\mathbb{D}$ endowed with the topology of uniform convergence on all compact subsets of $\mathbb{D}$. When one is interested in the boundary behaviour of a holomorphic function $f$ in $\D$, one can look at the behaviour on $\T$ of the partial sums of the Taylor expansion of $f$ at $0$, or one can look at the behaviour of $f(z)$ when $z$ in $\D$ approaches $\T$. In the second case, the notion of cluster set of $f$ at $z$ comes naturally into play. Let us recall that if $\gamma:[0,1)\to \D$ is a continuous path to a boundary point of $\D$ (that is $\gamma(r)\to z \in \T$ as $r\to 1$), the cluster set of $f$ along $\gamma$ is defined as
	\[
	C_{\gamma}(f):=\{w\in \C:\, \exists (r_n)_n\in [0,1),\,r_n\to 1,\,\text{ such that }f\circ \gamma(r_n)\to w\}.
	\]
	In complex function theory, it is of great interest to distinguish and study classes of holomorphic functions with a \emph{regular} boundary behaviour. \emph{Regular} boundary behaviour at a point of $\T$ can mean, for instance, \emph{convergence}, \emph{Ces\`aro summability} or \emph{Abel summability} of the Taylor expansion at this point. We recall that a function $f$ in $H(\D)$ is \emph{Abel summable} at $\zeta \in \T$ if the quantity $f(r\zeta)$, $r\in [0,1)$, has a finite limit as $r\to 1$. Note that in the latter case, the cluster set of $f$ along the radius $\{r\zeta:\,r\in[0,1)\}$ is reduced to a single value. Yet, it is now well understood that \emph{non-regularity} is a \emph{generic} behaviour. We say that a property is generic in a Baire space $X$ if the set of those $x\in X$ which satisfy this property contains a countable intersection of dense open sets. For example, it was already observed in 1933 that the set of functions whose cluster sets along any radius is equal to $\C$, is residual in $H(\D)$ \cite{Kierst}. In 2008, it was even shown that the set - denoted by $\UU_C(\D)$ - of those functions $f\in H(\D)$ satisfying the previous property along any (continuous) path $\gamma:[0,1)\to \D$ having some limit $\zeta \in \T$ at $1$ is residual in $H(\D)$ \cite{BernalCP}. More recently the first author exhibited in \cite{Charp} another residual set of functions in $H(\D)$ with a boundary behaviour at least as wild as those of $\UU_C(\D)$. Let $\varrho:=(r_n)_n \subset [0,1)$ be a sequence convergent to $1$ and denote by $\UU_A(\D,\varrho)$ the set of those $f\in H(\D)$ that satisfy the following property: for any compact set $K\subset \T$ different from $\T$, the set $\{f_{r_n}|_K:\,n\in \N\}$ is dense in the space $\CC(K)$ of all continuous functions on $K$ endowed with the uniform topology. Here, for $r\in [0,1)$, we denote by $f_r$ the function given by $f_r(z)=f(rz)$, $rz\in \D$. It was then observed in \cite{Charp} that $\UU_A(\D,\varrho)$ is residual in $H(\D)$ and that $\UU_A(\D,\varrho)\subset \UU_C(\D)$. At this point, we would like to introduce another natural class of universal functions, defined on the model of $\UU_A(\D,\varrho)$ but independent of the choice of a sequence $\varrho$. Precisely, let $\UU_A(\D)$ denote the set of all functions $f$ in $H(\D)$ satisfying that, given any compact set $K\subset \T$ different from $\T$ and any $h\in \CC(K)$, there exists $(r_n)_n\subset [0,1)$ tending to $1$ such that $f_{r_n}$ approximates $h$ uniformly on $K$. One can easily checks that
	\[
	\UU_A(\D,\varrho)\subset \UU_A(\D) \subset \UU_C(\D).
	\]
	In the whole paper, the elements of $\UU_A(\D,\varrho)$ will be referred to as \emph{$\varrho$-Abel universal functions}, and those of $\UU_A(\D)$ as \emph{Abel universal functions}. For the interested readers, we should mention the papers \cite{Bayae,CharpKos} where universal boundary phenomena for holomorphic functions in several complex variables are exhibited and \cite{Abak} where universal boundary behaviour for harmonic functions on trees are studied.
	
	For those who are familiar with the topic, the property enjoyed by the functions in $\UU_A(\D,\varrho)$ evokes that enjoyed by the well-studied \emph{universal Taylor series}. In 1996, Nestoridis proved that there exists a residual set $\UU(\D)$ of functions $f$ in $H(\D)$ such that for any compact set $K\subset \C\setminus \D$, with connected complement, the set $\{S_n(f)|_K:\, n\in \N\}$ of all partial sums of the Taylor expansion of $f$ at $0$ is dense in $\CC(K)$ \cite{NestorF}. Let us denote by $\UU(\D,\T)$ the set of those functions $f\in H(\D)$ satisfying the previous property not for any compact set $K\in \C\setminus \D$ with connected complement, but only for all compact sets $K\subset \T$ different from $\T$. Note that $\UU(\D)\subset \UU(\D,\T)$. The functions in $\UU(\D)$ were extensively studied from many points of view. We refer to \cite{Bay1,bgnp,CharMou,CostaTsi,CostaJung,Gar,GarMan,MNadv,MNP,MouJMAA} and the references therein for a non-exhaustive list of papers. In particular, some of these highlight the fact that functions in $\UU(\D)$ enjoy irregular boundary behaviour, for example along radii. In comparison, the recently introduced Abel universal functions were considered in two papers \cite{Charp,Maro}. In \cite{Charp}, the first author sketched a comparison of the sets $\UU(\D)$ and $\UU_A(\D,\varrho)$ and noticed that the main results of \cite{Cos,Gehlen} actually imply that none of them is included in the other. More precisely, the assertion $\UU_A(\D,\varrho)\setminus \UU(\D)\neq \emptyset$ was derived from the fact that the partial sums of the functions in $\UU(\D)$ have to behave wildy not only at the boundary of $\D$, but also on compact sets as far as possible from $0$. This thus leads to ask rather if $\UU(\D,\T)$ and $\UU_A(\D,\varrho)$ or $\UU_A(\D)$ are comparable. The question was explicitly stated in \cite{Charp}.
	
	The first goal of this paper is to show that neither $\UU_A(\D)$ nor $\UU_A(\D,\varrho)$ is comparable to $\UU(\D,\T)$ and that $\UU_C(\D)\setminus \UU_A(\D)\neq \emptyset$, making it interesting to study independently the functions in $\UU_A(\D,\varrho)$ or $\UU_A(\D)$. It is worth mentioning that, whereas exhibiting generic functions with universal behaviour is rather standard, on the contrary building \emph{non-generic} functions that still enjoy some universal property can be very tricky. This is probably why there are very few results of this type. For example, in order to build up a function in $\UU_A(\D,\varrho)$ which is not in $\UU(\D,\T)$, we will make use of a purely constructive trick. In passing, we will observe that this trick can turn out to be useful in order to attack one of the most important open question about universal series: does the derivative of a universal Taylor series remain a universal Taylor series? We will build functions in $\UU(\D,\T)$ with partial sums simultaneously enjoying a universal property outside $\overline{\D}$, and whose derivative is not even in $\UU(\D,\T)$. Let us observe that by Weierstrass' theorem, it is easily seen that if $\{S_n(f)|_K:\,n\in \N\}$ is dense in $\CC(K)$ for any $K\subset \C\setminus \overline{\D}$, then so is $\{S_n(f')|_K:\,n\in \N\}$. The same question makes sense for $\varrho$-Abel universal functions. In this context, we will be able to exhibit functions in $\UU_A(\D,\varrho)$ whose derivative is not in $\UU_A(\D,\varrho)$, answering a question also posed in \cite{Charp}.
	
	\medskip
	
	Being now convinced that the notions of universal Taylor series and Abel universal functions are quite distinct, it is legitimate to consider the study of the second one for itself. Like universal Taylor series, Abel universal functions are natural examples within the large theory of universality in operator theory. Let us say that, given two Fréchet spaces $X$ et $Y$, a family $(T_i)_{i\in I}$ of continuous operators from $X$ to $Y$ is \emph{universal} if there exists $x\in X$ such that the set
	\[
	\{T_i(x):\,i\in I\}
	\]
	is dense in $Y$. Such a vector $x$ is said to be universal for $(T_i)_{i\in I}$. Most of the concrete examples of universal families of operators fall into the situation where $I=\N$, $X=Y$ and $T_n=T^n$, $n\in \N$, where $T$ is a continuous operator from $X$ to $X$. In this case the operator $T$ is called \emph{hypercyclic} whenever $(T^n)_n$ is universal and the notion lies within linear dynamics, a very active branch of operator theory. Apart from this setting, there are some other \emph{natural} families of operators that are universal \cite{bgnp,GEbams}. By \emph{natural} families of operators, we mean families of operators which naturally come into play in mathematical analysis. This is in particular the case of the family $\TT_{\varrho}^K:=\{T_{\varrho,n}^K:\, n\in \N\}$ (resp. $\mathcal{S} ^K:= \{S_n^K:\,n\in \N\}$) defined, for a compact set $K\subset \T$ (resp. $K\subset \C\setminus \D$) by
	\[
	T_{\varrho,n}^K:\left\{\begin{array}{lll}H(\D) & \to & \CC(K)\\
	f & \mapsto & f_{r_n}|_K
	\end{array}\right.
	\quad (\text{resp. } S_n^K:\left\{\begin{array}{lll}H(\D) & \to & \CC(K)\\
	f=\sum_ka_kz^k & \mapsto & \sum_{k=0}^na_kz^k|_K
	\end{array}\right.),
	\]
	where $\varrho=(r_n)_n$ is a given sequence in $[0,1)$ tending to $1$. According to the definitions given in the first part of the introduction, the elements of $\UU_A(\D,\varrho)$ with $\varrho=(r_n)_n$ (resp. $\UU(\D)$) appear as functions in $H(\D)$ that are universal for all families $\TT_{\varrho}^K$, $K\subset \T$ different from $\T$ (resp. $\mathcal{S}^K$, $K\subset \C \setminus \D$ with connected complement). Let us observe that the sequences $(T_{\varrho,n}^K)_n$ are universal sequences of composition operators that fit within the framework of the recent paper \cite{CharMouComp}. Further, when $I$ is not countable, standard examples of universal families $(T_i)_{i\in I}$ are given by semigroups. Universality or hypercyclicity of semigroups has been a subject of interest during the last decade, and strong similarities with hypercyclicity of a single operator have been discovered, see \cite{CoMuPe} and \cite[Chapter 3]{BayMat}. In this context, Abel universal functions thus appear as singular and natural examples of objects that are universal for the non-semigroup families $\TT^K:=(T_r^K)_{r\in [0,1)}$ of operators, where $K$ is any subset of $\T$ different from $\T$, and $T_r^K$ is defined by
	\[
	T_r^K:\left\{\begin{array}{lll}H(\D) & \to & \CC(K)\\
	f & \mapsto & f_{r}|_K
	\end{array}\right..
	\]
	
	With all this in mind, it looks quite motivating to wonder which of the most interesting phenomena observed in linear dynamics can also be observed for the natural families of operators that are universal and do not fall into the concept of hypercyclicity. The remainder of this paper will focus on this program around the notion of Abel universal functions.
	
	One of the most elegant results on hypercyclicity, due to Bourdon and Feldman \cite{Bour}, asserts that a vector $x\in X$ is hypercyclic for a bounded operator $T:X\to X$ whenever its orbit under $T$ is \emph{somewhere} dense in $X$. A straight consequence of this fact is a result, independently obtained earlier by Costakis \cite{Costa} and Peris \cite{Peris}, telling us that if for finitely many $x_1,\ldots,x_l\in X$ the set $\bigcup_{k=1}^l\{T^n (x_k):\,n\in \N\}$ is dense in $X$, then one of the $x_i$'s is hypercyclic for $T$. Both results extend to the setting of semigroups \cite{BayMat}. One may now ask whether this property is also shared by sequences $\TT^K$ or $\TT_{\varrho}^K$, for some $\varrho=(r_n)_n \subset [0,1)$ tending to $1$ and $K\subset \T$ different from $\T$. In fact we will constructively prove that this is not the case: given $K_0\subset \T$ different from $\T$ and $\varrho=(r_n)_n$, there exist two functions $f_1,f_2\in H(\D)$, none of them Abel universal for $\TT^{K_0}$, such that the set $\{T_{\varrho,n}^{K}(f_i):\,n\in \N, i=1,2\}$ is dense in $\CC(K)$ for any compact set $K\subset \T$ different from $\T$. As a consequence, this will also show that the families $\TT^K$ and $\TT_{\varrho}^{K}$ do not satisfy a Bourdon-Feldman type property. We should mention that an analogue of those results was obtained by the second author for Fekete universal series, a \emph{real-analytic} variant of universal Taylor series \cite{MouArch}. However, the case of universal Taylor series - \emph{i.e.} for the sequences $\mathcal{S}^K$, $K\subset \C \setminus \D$ with connected complement - is still open.
	
	In 2006, Bayart and Grivaux \cite{BayGri} introduced the notion of \emph{frequent hypercyclicity}, which quickly became central in linear dynamics (see the books \cite{BayMat,gp} and the recent paper \cite{GMM}). It was extended to the larger setting of universality \cite{BonGe}. A sequence $(T_n)_n$ of continuous operators from a Fréchet space $X$ to another one $Y$ is said to be \emph{frequently universal} if there exists $x\in X$ such that the set $\{n\in \N:\, T_n(x) \in U\}$ has positive \emph{lower density} for any non-empty open set $U$ of $Y$. For the definition of the lower density, we refer to Section \ref{section-application}. Roughly speaking, it quantifies the proportion of elements in $\{n\in \N:\, T_n(x) \in U\}$ among all natural numbers. The same notion makes perfectly sense for a family $(T_i)_{i\in I}$ of operators when $I$ is an interval in $\R$, replacing the lower density by the \emph{uniform lower density} (see Section \ref{section-application} for the definition). As far as we know, non-discrete frequent universality was only considered, rather briefly, for semigroups \cite[Chapter 7]{gp}. On the contrary, several descriptions of frequently hypercyclic operators among classes of concrete operators have been obtained \cite{BayMat,gp}. In \cite{Char-countably,MouMun}, it was observed that there cannot exist functions in $H(\D)$ that are frequently universal for all sequences $\mathcal{S}^K$, $K\subset \C \setminus \D$ with connected complement. Yet, if $K$ is fixed outside $\D$, it is still open whether there exist functions in $H(\D)$ frequently universal for $\mathcal{S}^K$. In a broad but different setting, we mention that the authors of \cite{Abak} prove the existence of harmonic functions on trees with frequently universal boundary behaviour.
	
	The last section of this paper is devoted to showing various results revolving around frequent Abel universality. For example, given $\varrho =(r_n)_n\subset [0,1)$ tending to $1$, we will prove a rather strong result implying the existence of a dense meagre subset of $H(\D)$ consisting of functions that are frequently universal for all sequences $\TT_{\varrho}^K$, $K\subset \T$ different from $\T$ (resp. for all families $\TT^K$, $K\subset \T$ different from $\T$). 
	We will even show that there exist functions in $H(\D)$ that are frequently universal for all families $\TT_{\varrho(\alpha)}^K$, $K\subset \T$ different from $\T$, where $(\varrho(\alpha))_{\alpha \in A}$ is a countable collection of pairwise disjoint sequences in $[0,1)$ tending to $1$. These results thus lie within the newly active topic of common frequent universality, see \cite{Bay-com,CEMM}.
	
	\medskip
	
	The paper is organized as follows. The next section gathers some basic materials and definitions that will be used all along the paper. In Section \ref{sec-dist}, we prove that all the classical notions of universal functions are distinct. Section \ref{sec-deriv} deals with the non-stability of the classes of Abel universal functions under the action of the differentiation operator. In Sections \ref{multi} and \ref{section-application} we focus on the "dynamical" properties of the sequences $\TT_{\varrho}^K$ and $\TT_{\varrho}$ with respect to the concept of multi-universality and (common) frequent universality respectively.
	
	\section{Preliminaries}\label{prelim}
	
	The aim of this section is to introduce once and for all, the notations and objects that will be used often in the paper. We start with general notations. The letter $\N$ will stand for the set $\{0,1,2,3,\ldots\}$ of all non-negative integers. A sequence of general terms $u_n$, $n\in \N$, will be denoted by $(u_n)_n$. If $P$ is any polynomial, we will denote by $\text{deg}(P)$ its degree and by $\text{val}(P)$ its valuation. If $E\subset \C$, the notation $\text{int}(E)$ will stand for the set of all interior points of $E$. For any $r\geq 0$ and $E\subset \C$, we will denote by $rE$ the set $\{rz:\, z\in E\}$. For $a\in \C$ and $r\geq 0$, the open disc centred at $a$ with radius $r$ will be denoted by $D(a,r)$.
	
	\medskip

	Let us now focus on more specific notations. In the whole paper, $\varrho=(r_n)_n$ denotes an arbitrary sequence in $[0,1)$ converging to $1$. For $r\in [0,1)$ we denote by $\phi_r$ the function defined by $z\mapsto rz$, $z\in \mathbb{C}$. Without possible confusions, we will use the same notation to denote a function defined in a set $E\subset \C$ and its restriction to a subset of $E$. Given a compact set $K\subset \T$, different from $\T$, we will denote by $\VV(K)$ a countable set of non-empty open sets defining the topology of $\CC(K)$. 
	
	With these notations, the definition of Abel universal functions can be reformulated as follows.
	
	\begin{defi}\label{defi-Abel-univ-D}A function $f \in H(\D)$ is called $\varrho$-Abel universal if for any compact set $K\subset \T$ different from $\T$, and any $V\in \VV(K)$ the set
		\[
		N_{f}(V,K,\varrho):=\{n\in \N:\,f\circ \phi_n \in V\}
		\]
		is non-empty (or equivalently infinite), where $\phi_n:=\phi_{r_n}$.
	\end{defi}
	
	We recall that the set $\UU_A(\D,\varrho)$ of $\varrho$-Abel universal functions is a dense $G_{\delta}$-subset of $H(\D)$ \cite{Charp}. Similarly, the definition of a Abel universal functions given in the introduction is equivalent to the following one.
	
	\begin{defi}A function $f\in H(\D)$ is called Abel universal if for any compact set $K\subset \T$ different from $\T$, and any $V\in \VV(K)$, the set
		\[
		N_f(V,K):=\{r\in [0,1):\,f\circ \phi_r\in V\}
		\]
		is non-empty.
	\end{defi}
	Clearly, $\UU_A(\D,\varrho)\subset \UU_A(\D)$, where $\UU_A(\D)$ is the set of all Abel universal functions. Note also that if $f\in \UU_A(\D)$, then $1$ is a limit point of $N_f(V,K)$, for any $V$ and $K$.
	
	\medskip
	
	We introduce the following technical notations that will be used in most of the proofs:
	
	\begin{itemize}\item$(\vp_n)_n$ denotes an enumeration of all polynomials with coefficients whose real and imaginary parts are rational.
		\item $(K_n)_n$ denotes a sequence of compact subsets of $\T$, with connected complement, such that for any compact set $K\subset \T$, different from $\T$, there exists $n\in \N$ such that $K\subset K_n$ (see for e.g. \cite{Charp} for the existence of such sequence).
		\item $(\e_n)_n$ denotes a decreasing sequence of positive real numbers such that $\sum_n\e_n < 1$. The speed of decrease of $(\e_n)_n$ may change from a proof to another, and will be specified if necessary.
		\item $\alpha,\beta : \N \to \N$ denote two functions such that for any $n,m\in \N$, there exists an increasing sequence $(n_l)_l\subset \N$ such that $\alpha(n_l)=n$ and $\beta(n_l)=m$ for any $l\in \N$.
	\end{itemize}
	
	All these notations will be tacitly used throughout the paper.
	
	\medskip
	
	The purpose of the next sections is twofold: first, to compare the sets $\UU_A(\D,\varrho)$ and $\UU_A(\D)$ with other classical sets of analytic functions in $\D$ with universal behaviour at the boundary; second, to study Abel universal functions, in particular with respect to classical notions coming from linear dynamics.

	\section{Abel universal functions among universal holomorphic functions}\label{sec-dist}
	
	For $f=\sum_ka_kz^k \in H(\D)$ and $n\in \N$, we denote by $S_n(f)$ the $n$th partial sum of the Taylor expansion of $f$ at $0$, i.e. $S_n(f)=\sum_{k=0}^na_kz^k$.
	
	\medskip
	
	Let us recall more explicitly the definitions of the classes of universal holomorphic functions that we intend to compare with that of Abel universal functions. The first one consists of universal Taylor series.
	
	\begin{defi}[Universal Taylor series]$\,$
		\begin{enumerate}	
			\item We denote by $\UU(\D)$ the set of those functions $f \in H(\D)$ which satisfy the following property: for any compact set $K\subset \C\setminus \D$ with connected complement, and any function $\varphi$ continuous on $K$ and holomorphic in its interior, there exists an increasing sequence $(\lambda_n)_n$ of integers such that
			\[
			\sup_{\zeta \in K}\left\vert S_{\lambda_n}(f)(\zeta) -\varphi (\zeta) \right\vert \to 0 \quad \text{as }n\to \infty.
			\]
			The elements of $\UU(\D)$ are called \textit{universal Taylor series}.
			\item We denote by $\UU(\D,\T)$ the set of those functions $f \in H(\D)$ which satisfy the following property: for any compact set $K\subset \T$, different from $\T$, and any function $\varphi$ continuous on $K$, there exists an increasing sequence $(\lambda_n)_n$ of integers such that
			\[
			\sup_{\zeta \in K}\left\vert S_{\lambda_n}(f)(\zeta) -\varphi (\zeta) \right\vert \to 0 \quad \text{as }n\to \infty.
			\]
			The elements of $\UU(\D,\T)$ will be called \textit{$\T$-universal Taylor series}.
		\end{enumerate}
	\end{defi}
	
	As already said, the radial limit of a function holomorphic in $\mathbb{D}$ at a boundary point along an increasing sequence $(r_n)_n$ tending to $1$ can be seen as the operation on the Taylor partial sums at $0$ of this function by a regular process of summation. Indeed one can write $f(r_nz)=\sum_{k} a_kr_n^kz^k=\sum_{k} c_{n,k} S_k(f)$ with $c_{n,k}=r_n^k(1-r_n)$. One can check that this process of summation is regular (see \cite{Zyg}). Now, it turns out that the universality of a holomorphic function in $\D$ is often preserved by the action of a summation process. It is in particular the case that the Ces\`aro means of the Taylor partial sums of an analytic function in $\D$ are universal if and only if the function is itself a universal Taylor series \cite{Bay1}; see \cite{CharMou,MouMun} for more general summation processes. Let us then introduce the following definition.
	\begin{defi}[Ces\`aro universal Taylor series]$\,$
		\begin{enumerate}	
			\item We denote by $\UU_{Ces}(\D)$ the set of those functions $f \in H(\D)$ which satisfy the following property: for any compact set $K\subset \C\setminus \D$ with connected complement, and any function $\varphi$ continuous on $K$ and holomorphic in its interior, there exists an increasing sequence $(\lambda_n)_n$ of integers such that
			\[
			\sup_{\zeta \in K}\left\vert \frac{1}{\lambda_n+1}\sum_{j=0}^{\lambda_n}S_{j}(f)(\zeta) -\varphi (\zeta) \right\vert \to 0 \quad \text{as }n\to \infty.
			\]
			The elements of $\UU_{Ces}(\D)$ are called \textit{Ces\`aro universal Taylor series}.
			\item We denote by $\UU_{Ces}(\D,\T)$ the set of those functions $f \in H(\D)$ which satisfy the following property: for any compact set $K\subset \T$, different from $\T$, and any function $\varphi$ continuous on $K$, there exists an increasing sequence $(\lambda_n)_n$ of integers such that
			\[
			\sup_{\zeta \in K}\left\vert \frac{1}{\lambda_n+1}\sum_{j=0}^{\lambda_n}S_{j}(f)(\zeta) -\varphi (\zeta) \right\vert \to 0 \quad \text{as }n\to \infty.
			\]
			The elements of $\UU_{Ces}(\D,\T)$ are called \textit{$\T$-Ces\`aro universal Taylor series}.
		\end{enumerate}
	\end{defi}
	
	The last class of universal functions is that of functions with maximal cluster set along any path to $\T$ (by definition, a path to $\T$ will always stand for a \textit{continuous} function $\gamma:[0,1)\to \D$ such that $\gamma(r)\to z$ for some $z\in \T$).
	
	\begin{defi}[Functions with maximal cluster sets along every path]We denote by $\UU_C(\D)$ the set of those functions $f \in H(\D)$ which satisfy the following property: for any path to $\T$, the cluster set $C_{\gamma}(f)$ of $f$ along $\gamma$ is maximal, i.e. equal to $\C$.
	\end{defi}
	
	All the sets introduced above are residual in $H(\D)$, so that they intersect each other. It is natural to wonder whether some non obvious inclusions may hold. So far, here is what is known (see \cite{Bay1,Charp,CharMou} for e.g.): for any $\varrho=(r_n)_n$,
	
	\begin{itemize}
		\item $\UU_{Ces}(\D)=\UU(\D) \subsetneq \UU(\D,\T) \subset \UU_{Ces}(\D,\T)$;
		\item $\UU_A(\D,\varrho) \subset \UU_A(\D) \subset \UU_C(\D)$;
		\item $\UU(\D)\setminus \UU_A(\D)\neq \emptyset$ and $\UU_A(\D,\varrho)\setminus \UU(\D) \neq \emptyset$. 
	\end{itemize}
	Note that the inclusions $\UU(\D) \subset\UU(\D,\T)$ and $\UU_A(\D,\varrho) \subset \UU_A(\D) \subset \UU_C(\D)$ are either trivial or obvious. Moreover, that $\UU(\D)\setminus \UU_C(\D)\neq \emptyset$ is a consequence of the existence of universal Taylor series that are Abel summable at some boundary points \cite{Cos}. The assertion $\UU_A(\D,\varrho)\setminus \UU(\D)\neq \emptyset$ was observed in \cite{Charp} using that functions in $\UU(\D)$ possess Ostrowski-gaps, while some functions in $\UU_A(\D,\varrho)$ may not have Ostrowski-gaps. We mention that functions in $\UU(\D,\T)$ may not have Ostrowski-gaps in general \cite{CharMou}. The aim of this section is to contribute in completing the description of the relationships between these classes. Precisely, we will prove the following.
	
	\begin{prop}\label{distinction-classes}For any $\varrho=(r_n)_n$, we have
		\begin{enumerate}[(a)]
			\item $\UU_A(\D,\varrho)\setminus \UU_{Ces}(\D,\T)\neq \emptyset$;
			\item\label{of(b)} $\UU_A(\D)\setminus \UU_A(\D,\varrho)\neq \emptyset$;
			\item $\UU_C(\D)\setminus \UU_A(\D)\neq \emptyset$.
		\end{enumerate}
		In particular, the sets $\UU_A(\D,\varrho)$ and $\UU(\D,\T)$ are uncomparable.
	\end{prop}
	In passing, we can deduce (from (a) and $\UU_A(\D)\subset \UU_C(\D)$) that $\UU_C(\D)\setminus \UU_{Ces}(\D,\T)\neq \emptyset$. However, it is still an open question whether $\UU_{Ces}(\D,\T)$ is included in $\UU(\D,\T)$ or not (i.e. whether $\UU_{Ces}(\D,\T)=\UU(\D,\T)$ or not).

	\begin{proof}[Proof of Proposition \ref{distinction-classes}]
		
		We first prove (a). Let $(R_n)_n\subset [0,1)$ be such that $0<R_n<r_n<R_{n+1}<r_{n+1}<1$, $n\in\N$.
		Let us define, for every $0<r<1$ and $k\geq 0$,
		\begin{equation}\label{inequ_fond}
		H_k(r)=\sum_{j\geq k}h_j(r)\quad \hbox{ where }\quad h_j(r):=\sum_{i=j}^{+\infty}2ir^i .
		\end{equation}
		Observe that $H_k(r)\rightarrow 0$ as $k$ tends to infinity. 
		We choose $u_1>0$ such that 
		$$\max\left(H_{u_{1}}(r_{1}),\frac{(r_1/R_{2})^{u_1}}{1-r_1/R_{2}}\right)\leq \e_1.$$
		By Mergelyan's theorem we find $P_1=\sum\limits_{i\geq u_1}a_{i,1}z^i$ so that 
		$$\sup_{\vert z\vert\leq R_1}\vert P_1(z)\vert\leq \e_1\quad 
		\hbox{ and }\quad \sup_{z\in K_{\alpha(1)}}|P_1(r_1z)-\vp_{\beta(1)}(z)|\leq \e_1.$$
		Let us build by induction sequences $(P_n)_n$ and $(Q_n)_n$ of polynomials and an increasing sequence $(u_n)_n$ of integers as follows. Let $Q_0=0$ and suppose that $P_1,\dots,P_{n-1}$, $u_1,\dots,u_{n-1}$, and  $Q_0,\dots,Q_{n-2}$ have been chosen for $n\geq 2$. We shall write these polynomials under the following form 
		$$\hbox{for } k=2,\dots n-2,\quad P_k=\sum_{i=u_k}^{u_{k+1}-1}a_{i,k}z^i,\quad Q_{k}=\sum_{i=u_{k}}^{u_{k+1}-1}b_{i,k}z^i,\quad 
				\hbox{ with }u_k\geq\deg(P_{k-1})+1,$$
				and
				$$P_{n-1}=\sum_{i=u_{n-1}}^{\deg(P_{n-1})}a_{i,n-1}z^i,\hbox{ with }u_{n-1}\geq\deg(P_{n-2})+1,$$
			where the coefficients $a_{i,j}$ and $b_{i,j}$ will satisfy suitable conditions that will be specified in the induction.	
		We are going to construct $u_n$, $Q_{n-1}$ and $P_n$ in this order. First choose $u_n\geq 1+ \deg(P_{n-1})$ satisfying 	
		\begin{enumerate}
			\item \label{item3_c}$\max\left(H_{u_{n}}(r_{n}),
			\frac{(r_n/R_{n+1})^{u_n}}{1-r_n/R_{n+1}}\right)\leq \e_n$.
		\end{enumerate}
		Set $a_{i,n-1}=0$ for $i=\deg(P_{n-1})+1,\dots,u_n-1$ which implies  $P_{n-1}=\sum\limits_{i=u_{n-1}}^{u_n-1}a_{i,n-1}z^i.$ Let us consider  $w_{n-2}(z)=\sum\limits_{j=1}^{n-2}(P_j+Q_j)(z)$, $z\in \C$. Then we define $Q_{n-1}=\sum\limits_{i= u_{n-1}}^{u_n-1}b_{i,n-1}z^i$, where the coefficients $b_{i,n-1}$ for $u_{n-1}\leq i\leq u_n-1$, are built by induction as follows. We first set
		\[
		b_{u_{n-1},n-1}:=\left\{\begin{array}{ll}
		0 & \text{if } \vert \sum\limits_{l=1}^{u_{n-1}}S_l(w_{n-2})(1) + a_{u_{n-1},n-1}\vert \geq u_{n-1}\\
		2u_{n-1} & \text{otherwise}.
		\end{array}\right.
		\]
		Then, once $b_{u_{n-1},n-1},b_{u_{n-1}+1,n-1},\ldots, b_{k-1,n-1}$ have been built for some $u_{n-1}+1\leq k\leq u_n-1$, we set
		\begin{itemize}
			
			\item $b_{k,n-1}=0$ if we have \\
			$\left\vert \sum\limits_{l=1}^{k}S_l(w_{n-2})(1) + \sum\limits_{l=u_{n-1}}^{k-1}\sum\limits_{i=u_{n-1}}^{l}(a_{i,n-1}+b_{i,n-1}) + \sum\limits_{i=u_{n-1}}^{k-1}(a_{i,n-1}+b_{i,n-1}) + a_{k,n-1}\right\vert \geq k,$
			\item $b_{k,n-1}=2k$ otherwise.
		\end{itemize} 
		Doing so, we obtain for all $n\geq 2$ and for all $k=u_{n-1},\dots,u_n-1$,
		\begin{enumerate}\setcounter{enumi}{1} 
			\item \label{itemcoef1_c} $\vert \sum_{l=1}^{k}S_l(\sum_{j=1}^{n-1}(P_j+Q_j))(1) \vert  \geq k$;
			\item \label{itemcoef2_c} $|b_{k,n-1}|\leq 2k$.
		\end{enumerate} 
		Then apply Mergelyan's theorem to get $P_n=\sum\limits_{i\geq u_n}a_{i,n}z^i$ so that the following conditions hold:
		\begin{enumerate}\setcounter{enumi}{3}
			\item \label{item4_c} $\sup_{|z|\leq R_n} | P_n(z) | \leq \e_n$;
			\item \label{item6_c} $\sup_{z\in K_{\alpha(n)}}|P_n(r_nz)-\left(\vp_{\beta(n)}(z)-\sum_{0\leq j\leq n-1}(P_j+Q_j)(r_nz)\right)|\leq \e_n$.
		\end{enumerate}

		Finally, we set $f=\sum\limits_{n\geq 1}(P_n+Q_n)$ and get from \eqref{item4_c} and \eqref{itemcoef2_c} that $f\in H(\D)$. Moreover, it is clear from \eqref{itemcoef1_c} that $|\sum\limits_{l=1}^{k}S_l(f)(1)|\geq k$ for all $k\in \N$, and therefore $f\not\in \UU_{Ces}(\D,\T)$. The proof will be completed once we have proven that $f\in \UU_A(\D,\varrho)$. To do this, we will use the equality (\ref{inequ_fond}). Let us fix $n,m\in \N$ and let $(n_l)_l\subset \N$ be such that for any $l\in \N$, $\alpha(n_l)=n$ and $\beta(n_l)=m$. By \eqref{item6_c} we have, for any $\zeta \in K_n$,
		\[
		\left|f(r_{n_l}\zeta)-\vp_m(\zeta)\right| \leq \e_{n_l} + \left|Q_{n_l}(r_{n_l}\zeta)\right| + \left|\sum _{j\geq n_l+1} \left(P_j(r_{n_l}\zeta)+Q_j(r_{n_l}\zeta)\right)\right|.
		\]
		Moreover it follows from \eqref{itemcoef2_c}, (\ref{inequ_fond}) and \eqref{item3_c} that for any $\zeta \in K_{n}$,
		\begin{equation*}
		\left|Q_{n_l}(r_{n_l}\zeta)\right| \leq \sum_{i\geq u_{n_l}}2ir_{n_l}^i = h_{u_{n_l}}(r_{n_l})\leq H_{u_{n_l}}(r_{n_l})\leq \e_{n_l}.
		\end{equation*}
		Now, let us observe by \eqref{item4_c} and Cauchy's inequalities that we have $\left|a_{i,j}\right|\leq \e_{j}R_{j}^{-i}$ for any $j\geq 1$ and any $u_j \leq i \leq u_{j+1}-1$, whence by \eqref{item3_c}, (\ref{inequ_fond}) and \eqref{itemcoef2_c},
		
		\begin{eqnarray*}
			\left|\sum_{j\geq n_l+1} \left(P_j(r_{n_l}\zeta)+Q_j(r_{n_l}\zeta)\right)\right| & \leq & 
			\sum_{j\geq n_l+1}\sum_{i\geq u_{j}}\e_{j}\left(\frac{r_{n_l}}{R_{j}}\right)^i+ \sum_{j\geq n_l+1}\sum_{i\geq u_{j}}2ir_{n_l}^i
			\\& \leq & 
			\sum_{j\geq n_l+1}\e_{j}\left(\frac{r_{n_l}}{R_{n_l+1}}\right)^{u_{n_l}}\frac{1}{1-r_{n_l}/R_{n_l+1}}+
			\sum_{j\geq n_l+1} h_{u_j}(r_{n_l})
			\\&\leq & 
			\sum_{j\geq n_l+1}\e_{j}+H_{u_{n_l}}(r_{n_l})
			\\& \leq & \sum_{j\geq n_l+1}\e_{j} + \varepsilon_{n_l}.
		\end{eqnarray*}
		
		Altogether we get
		\[
		\sup_{\zeta\in K_n}\left|f(r_{n_l}\zeta) -\vp_m (\zeta) \right|\rightarrow 0\quad\text{as }l\to \infty,
		\]
		and thus $f\in \UU_A(\D,\varrho)$.
		
		\medskip

		\medskip
		
		Let us prove (b). Let $\varrho'=\left(r_n'\right)_n\subset [0,1)$ be an increasing sequence such that $r_n < r'_n < r_{n+1}$, $n\in \N$. Using Runge's theorem, we build by induction a sequence of polynomials $P_n$ such that
		\begin{enumerate}
			\item $\sup_{|z|\leq r'_{n-1}}\vert P_n(z)\vert\leq \varepsilon_n$;
			\item $\vert P_n(r_n) + \sum_{0\leq j\leq n-1} P_j(r_n)\vert\leq \varepsilon_n$;
			\item $\sup_{\zeta \in K_{\alpha(n)}}\left\vert P_n(r'_n\zeta) - (\varphi_{\beta(n)}(\zeta) - \sum_{0\leq j\leq n-1} P_j(r'_n\zeta))\right\vert\leq \varepsilon_n$.
		\end{enumerate}
		Then we set $f=\sum _{j\geq 0}P_j$. We get from (1) that $f\in H(\D)$. Let us now check that $f\in \UU_A(\D)\setminus \UU_A(\D,\varrho)$. It is first clear from (1) and (2) that $f\not \in \UU_A(\D,\varrho)$ since $|f(r_n)|\leq \sum_k\varepsilon_k$.
		Now, let us fix $n,m\in \N$ and let $(n_l)_l\subset \N$ be increasing such that for any $l\in \N$, $\alpha (n_l)=n$ and $\beta(n_l)=m$ for any $l\in \N$.
		Then, by (1) and (3), for any $l\in \N$ and any $\zeta \in K_{n}$,
		\[
		\left|f\left(r'_{n_l}\zeta\right) - \vp_m(\zeta) \right| \leq \e_{n_l} + \sum _{j\geq n_l+1}\left|P_j\left(r'_{n_l}\zeta\right)\right| \leq \sum _{j\geq n_l}\e_j,
		\]
		which goes to $0$ as $l\to \infty$. Thus $f\in \UU_A(\D,\varrho')\subset \UU_A(\D)$.
		
		\medskip
		
		
		We turn to proving (c). For $\zeta \in \T$ and $A\subset [0,1)$ we will denote by $A\zeta$ the set $\{r\zeta:\,r\in A\}$. Let us fix $\zeta_1\neq \zeta_2$ in $\T$, and let $(a_n)$ be an increasing sequence in $[0,1)$, tending to $1$. Note that for any $r\in [a_0,1)$ and any $K\subset \T$ containing $\zeta_1$ and $\zeta_2$, we have
		\[
		rK\cap \bigcup_n([a_{2n},a_{2n+1}]\zeta_1\cup [a_{2n+1},a_{2n+2}]\zeta_2) \neq \emptyset.
		\]
		We also fix a sequence $(\eta_n)$ of positive real numbers such that
		\[
		\eta_n < \min(a_{2n+2}-a_{2n+1},a_{2n+3}-a_{2n+2},|\zeta_1-\zeta_2|/(2\pi)),\quad n\in \N.
		\]
		Remark that $\eta_n \to 0$ as $n\to \infty$.
		Let us now consider, for $n\in \N$, the compact set $L_n$ defined  by
		\[
		L_n=\overline{D(0,a_{2n+2}+\eta_n)}\setminus D((a_{2n+2}+\eta_n)\zeta_1,2\eta_n).
		\]
		The sequence $(L_n)$ is clearly an exhaustion of $\D$ by compact sets with connected complement. Moreover, we have for any $n\in \N$ and any $m>n$,
		\[
		\text{int}(L_n) \supset [a_{2n},a_{2n+1}]\zeta_1 \cup [a_{2n+1},a_{2n+2}]\zeta_2\quad \text{and}\quad L_n\cap ([a_{2m},a_{2m+1}]\zeta_1\cup [a_{2m+1},a_{2m+2}]\zeta_2)=\emptyset.
		\]
		Then, for any $I\subset \T$, we denote by $C(I)$ the cone $\{r\zeta:\,r\in [0,1),\,\zeta \in I\}$ and by $C_n(I)$ the set $C(I)\cap \partial L_n$.
		Thus, for any $n\geq 1$, the sets $L_{n-1}$, $[a_{2n},a_{2n+1}]\zeta_1\cup [a_{2n+1},a_{2n+2}]\zeta_2$ and $C_n(K_{\alpha(n)})$ are pairwise disjoint (and with connected complement), so we can apply Runge's theorem to build a sequence $(P_n)$ of polynomials (with $P_0=0$), satisfying for any $n\geq 1$,
		\begin{enumerate}
			\item $\sup_{z\in L_{n-1}}\vert P_n(z)\vert\leq \varepsilon_n$;
			\item $\sup_{z\in [a_{2n},a_{2n+1}]\zeta_1\cup [a_{2n+1},a_{2n+2}]\zeta_2}\vert P_n (z) + \sum_{0\leq j\leq n-1} P_j(z)\vert\leq \varepsilon_n$;
			\item $\sup_{z\in C_n(K_{\alpha(n)})}\left\vert P_n(z) - (\varphi_{\beta(n)}(z/|z|) - \sum_{0\leq j\leq n-1} P_j(z))\right\vert\leq \varepsilon_n$.
		\end{enumerate}
		We set $f=\sum_{n\geq 0}P_n$. By (1), $f\in H(\D)$. Moreover $f\not \in \UU_A(\D)$. Indeed, let us consider the compact set $K=\{\zeta_1\}\cup\{\zeta_2\}$ and fix $r\in [0,1)$. Then there exists $n\in \N$ such that
		\[
		rK\cap ([a_{2n},a_{2n+1}]\zeta_1\cup [a_{2n+1},a_{2n+2}]\zeta_2)\neq \emptyset.
		\]
		So, by (1) and (2), there exists $i\in \{1,2\}$ such that $|f(r\zeta_i)|\leq \sum_k\e_k$. This proves that $f\not \in \UU_A(\D)$.
		
		It remains to check that $f\in \UU_C(\D)$. To do so, we will first prove that $f$ has some universality property at the boundary, with respect to $(\partial L_n)_n$. More precisely, let us fix $n,m\in \N$  and let $(n_l)_l\subset \N$ be increasing such that $\alpha(n_l)=n$ and $\beta(n_l)=m$ for any $l$. Then, by (1) and (3), for any $l\in \N$ and any $z \in C_{n_l}(K_n)$,
		\[
		\left|f\left(z\right) - \vp_m\left(\frac{z}{|z|}\right) \right| \leq \e_{n_l} + \sum _{j\geq n_l+1}\left|P_j(z)\right| \leq \sum _{j\geq n_l}\e_j,
		\]
		which goes to $0$ as $l\to \infty$.
		Now, let $\gamma$ be any continuous path to some $\zeta \in \T$ and let $c\in \C$. By construction, if $I\subset \T$ with $\zeta \in I$ is an arc, then $\gamma$ intersects all but finitely many $C_k(I)$. In particular, $\gamma$ intersects all $C_{n_l}(K_n)$, $l\in \N$, for some $n\in \N$ and some increasing sequence $(n_l)_l\subset \N$ so that
		\[
		\sup_{z\in \CC_{n_l}(K_n)}\left|f\left(z\right) - c \right| \to 0,\quad \text{ as } l\to \infty.
		\]
		Thus, if $(z_{n_l})_l$ is any sequence so that $z_{n_l}\in \gamma\cap C_{n_l}(K_n)$, we have $z_{n_l}\to \zeta$ and $f(z_{n_l})\to c$ as $l\to \infty$. This completes the proof that $f\in \UU_C(\D)$. 
	\end{proof}
	
	\begin{rem}{\rm
			The proof of (\ref{of(b)}) above shows that if $\varrho=(r_n)_n$ and $\varrho'=(r'_n)_n$ are such that $r_n=r'_n$ only for finitely many $n$, then the sets $\UU_A(\D,\varrho)$ and $\UU_A(\D,\varrho')$ are uncomparable.
		}
	\end{rem}

	\section{Derivatives of Abel universal functions}\label{sec-deriv}
	
	One of the most important open questions on universal Taylor series is to know whether $f'\in \UU(\D)$ whenever $f\in \UU(\D)$. In this section, we prove that the same question for Abel universal functions has a negative answer. Next, we also show that the answer to this question is "no" for a natural class of universal Taylor series which contains $\UU(\D)$ and is contained in $\UU(\D,\T)$. In the next results, we denote by $f^{(l)}$ the $l$th derivative of $f\in H(\D)$.
	
	\begin{thm}\label{derivatives-Abel}For any $\varrho$ and any $l\geq 1$, there exists $f\in \UU_A(\D,\varrho)$ such that $f^{(l)}\not\in \UU_A(\D,\varrho)$.
	\end{thm}
	
	In order to make the presentation simple, we shall extract the main ingredient of the construction as an independent lemma.
	
	\begin{lemme}\label{lem-deriv-Abel}Let $K$ be a compact subset of $\T$, different from $\T$, and let $\vp$ be a polynomial. For any $\e>0$, any $l\geq 1$ and any $r'<r\in [0,1)$, there exists a polynomial $P$ such that
	\begin{enumerate}
		\item \label{itemder1} $\sup_{|z|\leq r'}|P(z)|\leq \e$;
		\item \label{itemder2} $\sup_{[0,r']\cup[r,1)}|P^{(l)}(z)|\leq \e$;
		\item \label{itemder3} $\sup_{z\in rK}|P(z)-\vp(z)|\leq \e$
	\end{enumerate}
	\end{lemme}
	
	\begin{proof}If $rK\cap [0,1)=\emptyset$, it is enough to apply Mergelyan's theorem on $rK\cup \overline{V}\cup \overline{D(0,r')}$, where $V$ is any neighbourhood of $[0,1]$ with $\overline{V}\cap rK=\emptyset$, and then Weierstrass' theorem.
		
		Let us now focus on the case where $rK\cap [0,1)=\{r\}$ (note that this intersection cannot contain more than $1$ point). For $\delta>0$, let us denote $C_{\delta}:=\{z\in \C:\,\text{dist}(z,[r,1])\leq \delta\}$. Since $r'<r$ and by uniform continuity of $\vp$ on $rK$, there exist $\delta>0$ such that $\overline{D(0,r'+\delta)}\cap C_\delta=\emptyset$, and a function $\tilde{\vp}$ continuous on $\overline{D(0,r'+\delta)}\cup rK\cup C_{\delta}$, 
		such that
		\begin{equation}\label{eqderivlemm1}
		\sup_{z\in rK}|\vp(z)-\tilde{\vp}(z)|\leq \e/2 \quad \text{and} \quad \tilde{\vp}(z)= \left\{\begin{array}{lll} \vp(r) & \text{if } & z\in C_{\delta}\\
		0 & \text{if } & z\in \overline{D(0,r'+\delta)}
		\end{array}\right..
		\end{equation}
		Note that $\overline{D(0,r'+\delta)}\cup rK\cup C_{\delta}$ has connected complement and that $\tilde{\vp}$ is continuous on this set and holomorphic in its interior (since it is constant there). Now, by Mergelyan's theorem, there exists a sequence $(P_n)_n$ of polynomials tending to $\tilde{\vp}$ uniformly on $\overline{D(0,r'+\delta)}\cup rK\cup C_{\delta}$. In particular, by Weierstrass' theorem, $P^{(l)}_n\to 0$ uniformly on $[0,r']\cup [r,1]$, whence we conclude that there exists $n_0\in \N$ such that $P_{n_0}$ satisfies \eqref{itemder1}, \eqref{itemder2} and \eqref{itemder3} in the statement.
	\end{proof}
	
	\begin{proof}[Proof of Theorem \ref{derivatives-Abel}]Let us fix $l\geq 1$. We let $\varrho=(r_n)_n$ and set $r_{-1}=0$. Without loss of generality we shall assume that $(r_n)_n$ is increasing. Using Lemma \ref{lem-deriv-Abel}, we build by induction polynomial $P_n$, $n\in \N$, such that
		\begin{enumerate}
			\item $\sup_{|z|\leq r_{n-1}}\vert P_n (z) \vert \leq \e_n$;
			\item $\sup_{z\in [0,r_{n-1}]\cup[r_n,1)}\vert P^{(l)}_n(z) \vert \leq \e_n$;
			\item $\sup_{z \in K_{\alpha(n)}}\left\vert P_n (r_nz)-(\varphi_{\beta(n)}(z) - \sum_{0\leq j\leq n-1} P_j(r_nz)) \right\vert \leq \e_n$.
		\end{enumerate}
		By (1), $f:=\sum_nP_n \in H(\D)$ and, by (2), for any $n\in \N$ and any $k\geq 0$, $\vert P_n^{(l)}(r_k)\vert\leq \e_n$, so $\vert f^{(l)}(r_k)\vert=\sum_n \vert P_n^{(l)}(r_k)\vert \leq \sum_n \e_n <1$.
		In particular, $f^{(l)} \not \in \UU_A(\D,\varrho)$. The proof that $f\in \UU_A(\D,\varrho)$ is now quite standard. We fix $n,m \in \N$ and let $(n_k)_k\subset \N$ such that $\alpha(n_k)=n$ and $\beta(n_k)=m$ for any $k\in \N$. By construction, we have
		\begin{eqnarray*}\sup_{z\in K_n}|f(r_{n_k}z) - \vp_m(z)| & \leq & \e_{n_k} + \sum _{j>n_k} \vert P_j (r_{n_k}z)\vert  \\
			& \leq & \e_{n_k} +\sum _{j>n_k}\e_j
		\end{eqnarray*}
		which goes to $0$ as $k\to \infty$.
	\end{proof}
	
	\medskip
	
	Let us come back to the original open question whether the derivative of functions in $\UU(\D)$ are also in $\UU(\D)$. This question is still open but the following result asserts that there exist functions in $\UU(\D,\T)$ whose derivative is not in $\UU(\D,\T)$.
	
	\begin{thm}\label{derivU-D-T}For any $l\geq 1$, there exists $f\in \UU(\D,\T)$ such that $f^{(l)}\not \in \UU(\D,\T)$.
	\end{thm}
	
	\begin{proof}Let us first prove the result for $l=2$. Let $f=\sum_ka_kz^k \in \UU(\D,\T)$. For any $k$, let the real number $\alpha_k$ be such that for any $k\geq 0$,
		\[
		\Re((k+2)(k+1)(a_{k+2}+\alpha_{k+2}))=\max\{N\in \Z:\, N\leq \Re((k+2)(k+1)a_{k+2})\}.
		\]
		Then for any $k\geq 2$,
		\[
		\vert \alpha_k\vert \leq (k(k-1))^{-1}.
		\]
		In particular $\sum_k \alpha_k\zeta ^k$ is uniformly convergent for any $\zeta \in \T$, and thus the series $g:=\sum_k(a_k+\alpha_k)z^k$ still belongs to $\UU(\D,\T)$. However $g^{(2)}$ cannot be universal at $1$, since $S_n(g^{(2)})(1)\in \mathbb{Z}$ for any $n\in \N$.
		
		We can proceed similarly for any $l\geq 3$. The details are left to the reader. Now, if it were true that $f'\in \UU(\D,\T)$ whenever $f\in \UU(\D,\T)$, then $f^{(2)}$ would also belong to $\UU(\D,\T)$ for any $f\in \UU(\D,\T)$. That would contradict the first part of the proof.
	\end{proof}
	
	In fact, the previous result can be improved a little bit. Let us introduce another type of universal Taylor series. We will denote by $\D_R$ the open disc $D(0,R)$.
	
	\begin{defi}Let $R\geq 1$. We denote by $\UU(\D,R)$ the set of those functions $f\in H(\D)$ satisfying that, for any compact set $K_1\subset \C\setminus \overline{\D_R}$ and any compact set $K_2 \subset \overline{\D_R}\setminus \D$ both with connected complement, and for any function $\vp \in A(K_1\cup K_2)$, there exists an increasing sequence $(\lambda_n)_n\subset \N$ such that
		\[
		\sup_{z\in K_1 \cup K_2}|S_{\lambda_n}(f)(z) - \vp(z)|\to 0\quad \text{as }n\to \infty.
		\]
	\end{defi}
	
	Clearly, we have
	\[
	\UU(\D)=\bigcap_{R\geq 1}\UU(\D,R)\subset \UU(\D,R) \subset \UU(\D,\T)
	\]
	for any $R\geq 1$. Now, let us comment on the case $R=1$ which we think is of particular interest. Indeed, one can observe that
	\[
	\UU(\D,1)=\UU_0(\D)\cap \UU(\D,\T),
	\]
	where $\UU_0(\D)$ is the set consisting of those functions $f$ in $H(\D)$ that are universal for $\mathcal{S}^K$ for any compact set $K\subset \C\setminus \overline{\D}$ with connected complement. The difference between $\UU(\D)$ and $\UU_0(\D)$ is that universal approximation by partial Taylor sums of $f$ in $\UU_0(\D)$ are \emph{a priori}  valid only on compact subsets in $\C\setminus \D$ that do not touch $\T$. The class $\UU(\D)$ is obviously contained in $\UU_0(\D)$ and it is also well-known that both are distinct  (for e.g. \cite{MNP}). In fact, this class $\UU_0(\D)$ was already shown to be nonempty in the early 1970's by Chui and Parnes \cite{Chui} (see also \cite{Luh}). Using that the locally uniform convergence of a sequence of holomorphic functions implies the locally uniform convergence of the sequence of the derivatives, one can easily show that if $f$ belongs to $\UU_0(\D)$ then so does $f'$ (see also \cite[Proposition 3.5]{Costakis2}). Thus $\UU(\D,1)$ is the intersection of a class which is stable under differentiation and of a class that is not, according to Theorem \ref{derivU-D-T}. The following improvement tells us that the set $\UU(\D,R)$ is not stable under differentiation either, for any $R\geq 1$.
	
	\begin{thm}\label{derivative}\label{thm-derivative-UTS-R}For any $R\geq 1$ and any $l\geq 1$, there exists $f\in \UU(\D,R)$ such that $f^{(l)}\not \in \UU(\D,R)$.
	\end{thm}
	
	We will make use of the following lemma.
	
	\begin{lemme}\label{Runge-controle-ak}Let $R\geq 1$, $K\subset \C\setminus \overline{\D_R}$ with connected complement and $\vp \in A(K)$. For any $\e>0$, any $l\in \N$ and any $N\in\N$, there exists a polynomial $P=\sum_{k\geq N} a_k z^k$ with $|a_k|R^k\leq k^{-(l+2)}$ such that
		\[
		\sup_{|z|\leq R}\vert P(z) \vert\leq \e \quad \text{and}\quad \sup_{z\in K}\vert P(z)-\vp(z) \vert\leq \e.
		\]
	\end{lemme}
	
	\begin{proof}Let $\e$, $l$, $N$, $K$, $R$ and $\vp$ be fixed as in the lemma. Let $\eta>0$ be such that $0<\eta < \text{dist}(\D_R,K)$ and let $M\geq N$ be an integer such that $\frac{\e k^{l+2}R^k}{(R+\eta)^k}\leq 1$ for any $k\geq M$. Then Mergelyan's theorem gives the existence of a polynomial $P=\sum _{k\geq M}a_kz^k$ such that
		\[
		\sup_{|z|\leq R+\eta} \vert P(z) \vert\leq \e \quad \text{and}\quad \sup_{z\in K}\vert P(z)-\vp(z) \vert\leq \e.
		\]
		By Cauchy's inequalities, we deduce from the first inequality that for any $k\geq M$,
		\[
		|a_k|\leq \frac{\e}{(R+\eta)^k}
		\]
		This gives the conclusion by the choice of $M$.
	\end{proof}
	
	\begin{proof}[Proof of Theorem \ref{derivative}]
		Let us first fix $R\geq 1$ and $l\geq 2$ as in the statement. We keep the notations used in the previous proofs and let $(L_n)_n$ be an enumeration of compact sets in $(\overline{\D_R}\setminus \D)\times (\C\setminus \overline{\D_R})$, with connected complement, such that for any compact sets $K_1\subset \overline{\D_R}\setminus \D$ and any $K_2\subset \C\setminus \overline{\D_R}$, both with connected complement, there exists $n$ such that $K_1\cup K_2 \subset L_n$ (the proof of the existence of such sequence is straighforward and omitted). We also let $\psi_1:\N\to \N$, $\psi_2:\N\to \N$ and $\psi_3:\N\to \N$, such that for any $n_1,n_2,m \in \N$ there exists infinitely many $n\in \N$ such that $(\vp_{n_1},\vp_{n_2},L_m)=(\vp_{\psi_1(n)},\vp_{\psi_2(n)},L_{\psi_3(n)})$. Let $Q_0=P_0=0$ and $R_0=0$. We build by induction three sequences $(P_n=\sum a_k(n)z^k)_n$, $(Q_n=\sum b_k(n)z^k)_n$ and $(R_n)_n$ of polynomials (using Runge's theorem for $(P_n)_n$ and Lemma \ref{Runge-controle-ak} for $(Q_n)_n$), in the following order $Q_n\rightarrow P_n\rightarrow R_n$, with $R_n=\sum\limits_{k=\text{val}(P_n)}^{\deg(P_n)}\alpha_k z^k$, such that
		\begin{enumerate}
			\item $\text{val}(Q_n)> \text{deg}(P_{n-1})$;
			\item $\sup_{|z|\leq R}\vert Q_n(z) \vert \leq \e_n$;
			\item $\sup_{z\in L_{\psi_3(n)}\cap (\C\setminus \overline{\D_R})}\vert Q_n(z) - \left(\vp_{\psi_2(n)}(z) - \sum _{k=0}^{n-1} (Q_k + P_k + R_k)(z) \right)\vert \leq \e_n$;
			\item$|b_k(n)|\leq R^{-k}k^{-(l+2)}$ for any $n,k$;
			\item $\sup_{|z|\leq 1-1/n}\vert P_n(z) \vert \leq \e_n$;
			\item $\sup_{z\in L_{\psi_3(n)}\cap (\overline{\D_R}\setminus \D)}\vert P_n(z) - (\vp_{\psi_1(n)}(z) - \sum _{k=0}^{n-1} P_k(z))\vert \leq \e_n$;
			\item $\text{val}(P_n)> \text{deg}(Q_n)$;
			\item $\Re(k(k-1)\ldots(k-l+2)(k-l+1)R^{k-l}(a_{k}(n)+\alpha_{k}))$ is the integer part of $\Re(k(k-1)\ldots(k-l+2)(k-l+1)R^{k-l}a_{k}(n))$ for $k=\text{val}(P_n),\dots,\deg(P_n)$.
		\end{enumerate}
		
		We set $f=\sum _n Q_n +\sum _n P_n + \sum _{n} R_n$. Because of (2), (5) and (8), $f \in H(\D)$ and $\sum _{n} R_n$ is absolutely convergent on $\overline{\D_R}$ (indeed, (8) implies $|\alpha_k|\leq R^{-k}(k(k-1))^{-1}$).
		
		Moreover, (6) implies that $\{S_j(\sum_n P_n)|_K:\,j\in \N\}$ is dense in $A(K)$ for any $K\subset \overline{\D_R}\setminus \D$ with connected complement. By (8) and (4) respectively, $\sum _{n} R_n$ and the Taylor expansion at $0$ of $\sum _n Q_n$  are absolutely convergent on $\overline{\D_R}$, so $\{S_j(f)|_K:\,j\in \N\}$ is also dense in $A(K)$  for any $K\subset \overline{\D_R}\setminus \D$ with connected complement.
		
		Now, (3) and (7) imply that $\{S_j(f)|_K:\,j\in \N\}$ is dense in $A(K)$ for any $K\subset \C\setminus \overline{\D_R}$ with connected complement.
		
		Last, let us observe that by (8) the set $\{S_j\left((\sum_n P_n+\sum_n R_n)^{(l)}\right)(R):\,j\in \N\}$ is contained in $\Z$. Now, since by (4) the Taylor series of the $l$th derivative of $\sum Q_n$ is absolutely convergent on $\overline{\D_R}$, the set $\{S_j(f^{(l)})(R):\,j\in \N\}$ cannot be dense in $\C$. In particular, $f^{(l)}\not \in \UU(\D,R)$ and, by induction, there must be some $f\in \UU(\D,R)$ such that $f'\not\in \UU(\D,R)$.
	\end{proof}
	
	\medskip
	
	Since the important question whether $\UU(\D)$ is invariant under differentiation is still open, the last result makes the following question interesting.
	\begin{quest}Does the set $\UU(\D)$ coincide with $\UU(\D,R)$ for some $R\geq 1$?
	\end{quest}

	\section{Abel multi-universality}\label{multi}
	
	In this section, we are interested in the concept of multi-universality which comes from that of multi-hypercyclicity. Given a Fréchet space $X$ and a continuous operator $T:X\rightarrow X$, a finite subset $\{x_1,\dots,x_l\}$ of $X$ is said to be multi-hypercyclic for $T$ if the set $\bigcup_{k=1}^l\{T^nx_k:n\in\mathbb{N}\}$ is dense in $X$. Costakis \cite{Costa} and Peris \cite{Peris} independently proved that if $\{x_1,\dots,x_l\}$ is multi-hypercyclic for $T$ then one of the $x_i$'s is hypercyclic for $T$. Let us extend this notion to the setting of universality.

	\begin{defi}\label{defi-CP-1}
		Let $X$ and $Y$ be two Fréchet spaces and let $T_n:X\rightarrow Y$, $n\in\mathbb{N}$, be continuous linear maps. A finite set	$\{x_1,\dots,x_l\}$ of $X$ is said to be multi-universal for $(T_n)_n$ if $\bigcup_{k=1}^l\{T_nx_k:n\in\mathbb{N}\}$ is dense in $Y$.
	\end{defi}

	The concept of multi-universality was recently studied in the case of Fekete universal series \cite{MouArch}. In this section, we are going to investigate it for Abel universal functions. For $K\subset \T$ and $\varrho =(r_n)_n \subset [0,1)$ increasing and converging to $1$, let us denote as in the introduction by $\TT_{\varrho}^K$ the sequence $(T_{\varrho,n}^K)_n$ of continuous operators from $H(\D)$ to $\CC(K)$ defined, for any $n\in \N$, by $T_{\varrho,n}^K(f)=f_{r_n}|_K$.
	\begin{defi}\label{defi-CP-2}A set $\{f_1,\dots,f_l\}\in H(\D)$ is said to be \textit{$\varrho$-Abel multi-universal} if, for any $K\subset \T$ different from $\T$, it is multi-universal for the sequence $\TT_{\varrho}^K$.
	\end{defi}
	Equivalently, $\{f_1,\dots,f_l\}\in H(\D)$ is \textit{$\varrho$-Abel multi-universal} if for any compact set $K\subset\mathbb{T}$, different from $\mathbb{T}$, and any function $h\in \CC(\T)$, there exist $i\in\{1,\dots,l\}$ and an increasing sequence $(\lambda_n)_n\subset \N$ such that
	\[
	\sup_{z\in K}\vert f_i(r_{\lambda_n}z)-h(z)\vert \rightarrow 0 \quad \text{as } n\rightarrow \infty.
	\]
	In the same spirit, we can introduce the following definition.
	\begin{defi}A set $\{f_1,\dots,f_l\}\in H(\D)$ is said to be \textit{Abel multi-universal} if for any compact set $K\subset\mathbb{T}$, different from $\mathbb{T}$, and any function $h\in \CC(\T)$, there exist $i\in\{1,\dots,l\}$ and an increasing sequence of positive real numbers $(r_n)_n$ tending to $1$ such that $\sup_{z\in K}\vert f_i(r_{n}z)-h(z)\vert \rightarrow 0,$ as $n\rightarrow +\infty$.
	\end{defi}
	The following question naturally arises:
	\begin{quest}\label{quest-CP}If $\{f_1,\dots,f_l\}$ is $\varrho$-Abel multi-universal (resp. Abel multi-universal), does there exist $i\in\{1,\dots,l\}$ such that $f_i$ is a $\varrho$-Abel universal functions (resp. Abel universal functions)?
	\end{quest}
	We shall see that this question has a negative answer. To do this, inspired by \cite{MouArch}, we first construct a function $f$ in $H(\mathbb{D})$ such that for any compact set $K\subset \T$ containing $1$ and different from $\T$, the following two conditions hold:
	\begin{enumerate}[(a)]\item $T_{\varrho,n}^K(f)(1)$ is off some non-empty disc;
		\item the closure of $\{T_{\varrho,n}^K(f)|_K:\,n\in \N\}$ in $\CC(K)$ contains all continuous function in $\CC(K)$ whose value at $1$ is off some other non-empty disc.
	\end{enumerate}
	More precisely, our key proposition states the following.

	\begin{prop} \label{Prop_CP2} Let $\varrho=(r_n)_n\subset [0,1)$ be an increasing sequence tending to $1$ and let $a\in\mathbb{C}$. There exists a function $f \in H(\mathbb{D})$ which satisfies the following two properties:
		\begin{enumerate}
			\item for any $r_1<r<1$, $f(r)\notin D(a,1/2)$;
			\item for any compact set $K\subset \mathbb{T}$ containing $1$ and different from $\mathbb{T}$, and any function $h$ continuous on $K$, with $ h (1) \notin \overline{D(a,1)}$, there exists an increasing sequence $(\lambda_n)_n$ of integers such that
			\[
			\sup_{\zeta \in K}\left\vert f\left(r_{\lambda_n}\zeta \right) -h (\zeta) \right\vert \to 0 \quad \hbox{ as }n\to \infty.
			\]
		\end{enumerate}
	\end{prop}
	
	\begin{proof} We assume that $\sum_n\varepsilon_n< 1/4$. Let us set, for any $n\in \N$, $I_n=K_n \cup\{1\}$. Note that $I_n$ is different from $\T$ for any $n\in \N$ and that any compact subset of $\T$ containing $1$ and different from $\T$ is contained in some $I_n$. Let us also denote by $(\tilde{\varphi}_n)_n$ an enumeration of all polynomials with rational real and imaginary parts coefficients with value at $1$ off $D(a,1)$. Note that every continuous function on $\T$ with value at $1$ off $\overline{D(a,1)}$ can be uniformly approximated on every compact subset of $\T$ containing $1$ and different from $\T$, by a subsequence of $(\tilde{\varphi}_n)_n$.
		
		Set $f_1(z)=\tilde{\varphi}_{\alpha(1)}\left(\frac{z}{r_{1}}\right)$. Observe that $f_1(r_1)\notin D(a,1)$. Moreover we can assume that, for all $k\geq 2$, $\varepsilon_k<d(\tilde{\varphi}_{\alpha(k)}(1),D(a,1))$, where $d$ is the usual distance on $\mathbb{C}$. Using Mergelyan's theorem, we build by induction a sequence of polynomials $f_k$, $k\geq 2$, such that for any $k\geq 2$,
		\begin{enumerate}[(a)]
			\item \label{equ_cons1_cp2}
			$\sup_{z\in r_{k}I_{\beta(k)}}\left\vert f_{k}(z)-\tilde{\varphi}_{\alpha(k)}\left(\frac{z}{r_k}\right)
			\right\vert<\varepsilon_k$;
			\item \label{equ_cons2_cp2} 
			$\sup_{\vert z\vert \leq r_{k-1}}\left\vert f_{k}(z)-f_{k-1}(z)
			\right\vert<\varepsilon_k$;
			\item \label{equ_cons3_cp2}
			$\sup_{r\in [r_{k-1},r_k]}\vert f_{k}(r)-h_k(r)\vert<\varepsilon_k$, where $h_k:[r_{k-1},r_k]\rightarrow\mathbb{C}$ is a continuous functions such that $h_k(r_{k-1})=f_{k-1}(r_{k-1})$, 
			$h_{k}(r_k)=\tilde{\varphi}_{\alpha(k)}(1)$ and for all $r_{k-1}\leq r\leq r_{k},$ $h_k(r)\notin D(a,1)$.
		\end{enumerate}
	Observe that the choice of the sequence $(\varepsilon_k)_k$ and the property (\ref{equ_cons1_cp2}) ensure that $f_k(r_k)\notin D(a,1)$.\\
		Let us set 
		\[
		f(z)=f_1(z)+\sum_{k=1}^{+\infty}(f_{k+1}(z)-f_k(z)).
		\]
		From \eqref{equ_cons2_cp2} we deduce that $f\in H(\mathbb{D})$. Now, by \eqref{equ_cons1_cp2} and \eqref{equ_cons2_cp2}, we have for any $n\in \N$ and any $z\in I_{\beta(n)}$,
		\begin{equation}\label{equ_cons4_cp2}
		\begin{array}{lll}
		\left\vert f(r_{n}z)-\tilde{\varphi}_{\alpha(n)}\left(z\right)
		\right\vert & \leq & \displaystyle
		\left\vert f_n(r_{n}z)-\tilde{\varphi}_{\alpha(n)}\left(z\right)\right\vert+
		\sum_{k=n}^{+\infty}\left\vert f_{k+1}(r_{n}z)-f_k(r_{n}z)\right\vert\\
		& \leq & \displaystyle \left\vert f_n(r_{n}z)-\tilde{\varphi}_{\alpha(n)}\left(z\right)\right\vert+
		\sum_{k=n}^{+\infty}\sup_{\vert z\vert \leq r_{k}}\left\vert f_{k+1}(z)-f_k(z)\right\vert\\
		& < & \displaystyle \sum_{k=n}^{+\infty}\varepsilon_k.\end{array}
		\end{equation}
		This yields the point (2) of the proposition. Indeed, if $n,m\in \N$ are fixed and $(n_l)_l\subset \N$ is an increasing sequence such that $\alpha(n_l)=n$ and $\beta(n_l)=m$ for any $l\in \N$, then by \eqref{equ_cons4_cp2} $f_{r_{n_l}}$ tends to $\tilde{\varphi}_{n}$ uniformly on $I_m$ as $l\to \infty$.
		
		For the proof of (1), let us fix $r\in [r_1,1)$ and let $n\geq 2$ be an integer such that $r_{n-1} \leq r\leq r_n$.
		Since $\sum_k\varepsilon_k< 1/4$ we deduce from \eqref{equ_cons3_cp2} that
		\[
		\begin{array}{lll}\vert f(r)-h_n(r)\vert &\leq &\vert f(r)-f_n(r)\vert+\vert f_n(r)-h_n(r)\vert\\&\leq&
		\displaystyle\sum_{k\geq n}\vert f_{k+1}(r)-f_k(r)\vert +\vert f_n(r)-h_n(r)\vert\\
		&<&\displaystyle\frac{1}{4} + \frac{1}{4}=\frac{1}{2}.\end{array}
		\]
		We conclude the proof thanks to the fact that $h_n(r)\in \C\setminus D(a,1)$ for any $r_{n-1}\leq r\leq r_{n}$.
	\end{proof}
	
	We can easily deduce the answer to Question \ref{quest-CP}.
	
	\begin{thm}\label{Thm_CP} There exist two functions $f_1$ and $f_2$ in $H(\mathbb{D})$ such that the family $\{f_1,f_2\}$ is $\varrho$-Abel multi-universal (resp. Abel multi-universal) but neither $f_1$ nor $f_2$ is $\varrho$-Abel universal (resp. Abel universal).
	\end{thm} 
	
	\begin{proof} Let $a_1,a_2\in\mathbb{C}$ such that $\vert a_1- a_2\vert >2$. Let $f_1$ and $f_2$ in $H(\D)$ be given by Proposition \ref{Prop_CP2} applied respectively with $a=a_1$ and $a=a_2$. Clearly neither $f_1$ nor $f_2$ is Abel universal.
		Now, since $\overline{D(a_1,1)}\cap \overline{D(a_2,1)}=\emptyset$, it is also clear that 
		the family $\{f_1,f_2\}$ is $\varrho$-Abel multi-universal. To complete the proof, it is enough to observe that any $\varrho$-Abel multi-universal set is Abel multi-universal.
	\end{proof}
	
	\begin{rems}{\rm 
			(1) In fact, Proposition \ref{Prop_CP2} implies a result slightly stronger than Theorem \ref{Thm_CP}. Indeed, it shows the existence of a set $\{f_1,f_2\}$ in $H(\D)\setminus \UU_A(\D,\varrho)$ such that the following holds: for any function $h$ continuous on $\T$ there exists $i\in\{1,2\}$ such that for any compact set $K\subset\mathbb{T}$, different from $\mathbb{T}$, there exists 
			an increasing sequence $(\lambda_n)_n \subset \N$ so that $T_{\varrho,\lambda_n}^K(f) \to h$ in $\CC(K)$ as $n\to \infty$. Note that the latter property is stronger than that required in Definitions \ref{defi-CP-1} and \ref{defi-CP-2}.
			
			\medskip
			
			(2) Assume that a set $\{f_1,f_2\} \in H(\D)$ satisfies the following property: for any compact set $K\subset\mathbb{T}$ different from $\mathbb{T}$, there exists $i\in\{1,2\}$ for which
			$f_i$ is universal for $\TT_{\varrho}^K$. Note that this condition is stronger than that given in (1) and that there exist functions universal for $\TT_{\varrho}^K$ which are not Abel universal. Then such a property now implies that at least one of the two functions $f_1$ or $f_2$ is $\varrho$-Abel universal. Similarly, if we assume that for any compact set $K\subset\mathbb{T}$ different from $\mathbb{T}$, there exists $i\in\{1,2\}$ for which there exists $\varrho$ such that $f_i$ is universal for $\TT_{\varrho}^K$, then $f_1$ or $f_2$ is Abel universal.
			
			Let us only outline the proof in the case where $\varrho$ is fixed in advance. Thus assume by contradiction that $\{f_1,f_2\}$ satisfies the above property whereas neither $f_1$ nor $f_2$ is $\varrho$-Abel universal.
			Then, for each $i=1,2$, one can find a compact set $L_i\subset \mathbb{T}$, different from $\mathbb{T}$, a polynomial $h_i$ and $\varepsilon_i>0$ such that for all positive integers 
			$n$ there exists $z_n^{(i)}\in L_i$ satisfying
			\begin{equation}\label{eqqq}\vert f_i(r_nz_n^{(i)})-h_i(z_n^{(i)})\vert >\varepsilon_i.
			\end{equation}
			Using the fact that $L_1$ and $L_2$ are compact sets, there 
			exist increasing sequences of positive integers $(\lambda_n^{(i)})$ and $z_{\infty}^{(i)}\in L_i$, $i=1,2$, so that $z_{\lambda_n^{(i)}}^{(i)}\rightarrow z_{\infty}^{(i)}$. 
			We set, for $i=1,2,$ $K_i=\{z_{\infty}^{(i)}\}\cup\{z_{\lambda_n^{(i)}},n\geq 1\}$. Clearly $K_1$ and $K_2$ are compact subsets of $\mathbb{T}$ with $K_1\cup K_2\ne\mathbb{T}$. So by assumption one of the two functions $f_i$ is universal for $\TT_{\varrho}^{K_1\cup K_2}$, a contradiction with \eqref{eqqq}.
		}
	\end{rems}
	
	\medskip
	
	To conclude this section, let us recall that one of most elegant results in linear dynamics asserts that a vector $x$ in some separable Fréchet space $X$ is hypercyclic for some continuous operator $T:X\to X$ whenever the orbit $\{T^n x:\, n\in \N\}$ of $x$ under $T$ is \emph{somewhere} dense in $X$ (see \cite{Bour} or \cite[Theorem 6.5]{gp}). Costakis and Peris' result is a corollary of this result. The following proposition, which is a direct consequence of Proposition \ref{Prop_CP2}, shows that Bourdon-Feldman's theorem does not extend to the concept of universality.
	
	\begin{prop}For any $\varrho$, there exists a function $f\in H(\D)$ that is $\TT_{\varrho'}^{\{1\}}$-universal for no sequence $\varrho'$, and that enjoys yet the property that given any compact set $K\subset \mathbb{T}$ different from $\mathbb{T}$, the set $\{T_{\varrho,n}^K(f)|_K:\,n\in \N\}$ is somewhere dense in $\CC(K)$.
	\end{prop}

\section{Abel universality: a very flexible notion}\label{section-application}

Paraphrasing Section \ref{prelim}, a function $f\in H(\D)$ is Abel universal (resp. $\varrho$-Abel universal) if the set $N_{f}(V,K)$ (resp. $N_{f}(V,K,\varrho)$) is non-empty (or equivalently infinite) for any compact set $K\subset \T$, different from $\T$. It is natural to think of quantifying the size of these sets or, more generally, of understanding which "type" of sets they can be. The purpose of this section is to show that the notion of Abel universality is as flexible as one may think of, namely that the sets $N_{f}(V,K)$ can literally be "anything which makes sense". In this context, "anything which makes sense" will refer to a \emph{locally finite} countable family of compact subsets of $[r(V,K),1)$ for some $r(V,K)\in [0,1)$ close enough to $1$ (depending only on $V$ and $K$) and where, by definition, a countable family $\FF$ of compact subsets of $[0,1)$ is \emph{locally finite} if for any compact set $K\subset [0,1)$, $K$ intersects at most finitely many sets of $\FF$.

For $l\in \N$, let $(V_n(l))_n$ be a sequence of non-empty open \textit{balls} defining the topology of $\VV(K_l)$.

\begin{prop}\label{key-ingredient-FU}Let $\{\Gamma(l,n,m):\, l,n,m\in \N\}$ be a locally finite infinite family of pairwise disjoint segments in $[0,1)$. Set
	\[
	\Gamma(l,n):=\bigcup_{m\in \N}\Gamma(l,n,m), \quad l,n\in \N.
	\]
	Let $\varepsilon >0$, $g\in H(\D)$ and let $D$ be a closed disc centred at $0$. 
	There exist a family $\{r(l,n):\,l,n\in \N\}$ in $[0,1)$ and a function $f\in H(\D)$ such that $\sup_D |f - g| < \varepsilon$ and such that for any $l,n\in \N$ and any $r\in \Gamma(l,n)\cap [r(l,n),1)$, $f\circ \phi_r \in V_n(l)$.
\end{prop}

This proposition is of interest if for some $(l,n)$, the point $1$ is contained in the closure of $\Gamma(l,n)$ (if there is no $(l,n)$ for which it is the case, then $r(l,n)$ can be chosen so that $\Gamma(l,n)\cap [r(l,n),1)$ is empty). For our purpose, we are interested in applications where the sets $\Gamma(l,n)$ are "big" near $1$, for e.g. when $1$ is a limit point of every set $\Gamma(l,n)$, $l,n\in \N$.

\begin{proof} By Mergelyan's theorem we may and shall assume that for any $l,n\in \N$, $V_n(l)$ is centred at some entire function $g_n$ and has radius $\delta_n$, where $(\delta_n)_n$ is some sequence of positive numbers. By uniform continuity, for any $l,n\in \N$, we can define $r(l,n)\in [0,1)$ such that
\begin{equation}\label{eq-supp-sec-6}
\sup_{r\in [r(l,n),1]}\sup_{\zeta\in K_l}|g_n\circ \phi_r(\zeta)-g_n(\zeta)| < \frac{\delta_n}{4}.
\end{equation}
We shall also assume that $r(l,n)$ is larger than the radius of $D$ for any $l,n\in \N$. Let $\widetilde{V_n(l)}$ denote the ball of $\VV(K_l)$ with the same center as $V_n(l)$ with radius half of that of $V_n(l)$. For any $l,n,m\in \N$, let us denote by $\tilde{\Gamma}(l,n,m)$ the set $\Gamma(l,n,m)\cap [r(l,n),1)$ and by $\tilde{\Gamma}(l,n)$ the set
\[
\tilde{\Gamma}(l,n):=\Gamma(l,n)\cap[r(l,n),1) = \bigcup_{m\in \N}\tilde{\Gamma}(l,n,m).
\]
It may happen that the family $\{\tilde{\Gamma}(l,n,m):\, l,n,m\in \N\}$ is finite. This case is much simpler than that where it is infinite and is omitted. Let us then assume that the family $\{\tilde{\Gamma}(l,n,m):\, l,n,m\in \N\}$ is infinite. Also, upon removing all the empty sets $\tilde{\Gamma}(l,n,m)$ and reordering the remaining non-empty ones, we may and shall assume that every $\tilde{\Gamma}(l,n,m)$, $l,n,m \in \N$, is non-empty. Since $\{\tilde{\Gamma}(l,n,m):\, l,n,m\in \N\}$ is locally finite, we can then find an ordering $(\tilde{\Gamma}_i)_{i\in \N}$ of it and an increasing sequence $(s_i)_{i\geq 1}$ in $(0,1)$, tending to $1$, such that for any $i\geq 1$,
\[
\max(\tilde{\Gamma}_{i-1}) < s_i < \min(\tilde{\Gamma}_{i}).
\]
Note that for any $i\in \N$, there exists a unique pair $(l,n)$ such that $\tilde{\Gamma}_{i}\subset \tilde{\Gamma}(l,n)$. Let us write $D_0=D$ and denote by $D_i$, $i\geq 1$, the closed disc centred at $0$, with radius $s_i$. By induction, let us build a sequence $(P_i)_i$ of polynomials as follows. If $D_0=\emptyset$, set $P_0=0$. If not, let us apply Runge's theorem to define $P_0$ as any polynomial satisfying $\sup_{D_0} |P_0-g| < \varepsilon/2$ and
\[
\sup_{\zeta\in K_l}|P_0\circ \phi_r(\zeta)-g_n\circ \phi_r(\zeta)| < \frac{\delta_n}{4}
\]
for any $r\in \tilde{\Gamma}_0$, where $(l,n)$ is the only pair such that $\tilde{\Gamma}_{0}\subset \tilde{\Gamma}(l,n)$. Then we build by induction positive real numbers $\eta_k$, $k\geq 1$, and by applying Runge's theorem, polynomials $P_k$, $k\geq 1$, such that for any $j\geq 1$ and any $0\leq i\leq j-1$:
	\begin{enumerate}[(a)]
		\item $\sup_{K_l}\left\vert \sum _{k=0}^i P_k\circ \phi_r - g_n\circ\phi_r \right\vert  + \sum_{k=i+1}^j \eta_k < \frac{\delta_n}{4}$ for any $r \in \tilde{\Gamma}_{i}$ where $(l,n)$ is the only pair for which $\tilde{\Gamma}_{i}\subset \tilde{\Gamma}(l,n)$;
		
	\item $\eta_j < \varepsilon/2^{j+1}$;
		\item $\sup_{K_l}\left\vert \sum _{k=0}^j P_k\circ \phi_r - g_n\circ\phi_r \right\vert < \frac{\delta_n}{4}$ for any $r \in \tilde{\Gamma}_{j}$ where $(l,n)$ is the only pair for which $\tilde{\Gamma}_{j}\subset \tilde{\Gamma}(l,n)$;
		\item $\sup _{D_{j}}|P_j|< \eta _{j}$.
	\end{enumerate}
	
	Then we set $f=\sum _{i\geq 0} P_i$. By (b) and (d) and the choice of $P_0$, we get $f\in H(\D)$ and $\sup_{D} |f - g| < \varepsilon$. Let us now check that $f$ satisfies $f\circ \phi_r \in V_n(l)$ for any $l,n\in \N$ and any $r\in \tilde{\Gamma}(l,n)$. Fix $l,n\in \N$ and $r\in \tilde{\Gamma}(l,n)$. There exists a unique $i\in \N$ such that $r\in \tilde{\Gamma}_i$. We deduce from (a), (c) and \eqref{eq-supp-sec-6} that
	\[
	f\circ \phi_r = \sum_{k=0}^i P_k\circ \phi_r + \sum _{k\geq i+1}P_k\circ \phi_r \in \overline{\widetilde{V_n(l)}}\subset V_n(l).
	\]
	This gives the desired conclusion.
\end{proof}

\medskip

\begin{rem}{\rm If we assume that the sets $\Gamma(l,n,m)$ all have empty interior, then Proposition \ref{key-ingredient-FU} may be a little bit strengthened, using Mergelyan's theorem instead of Runge's theorem. More precisely, the reader may check that the following holds:
\begin{prop}Let $\{\Gamma(l,n,m):\, l,n,m\in \N\}$ be a locally finite infinite family of pairwise disjoint segments in $[0,1)$. Set
\[
\Gamma(l,n):=\bigcup_{m\in \N}\Gamma(l,n,m), \quad l,n\in \N.
\]
Let $\varepsilon >0$, $g\in H(\D)$ and let $D$ be a closed disc centred at $0$. 
There exists a function $f\in H(\D)$ such that $\sup_D |f - g| < \varepsilon$ and such that for any $l,n\in \N$ and any $r\in \Gamma(l,n)$, $f\circ \phi_r \in V_n(l)$.
\end{prop}
This statement makes it appear that the sequence $T^K_{\varrho,n}: f\in H(\D) \mapsto f\circ \phi_{r_n} \in \CC(K)$ is \emph{Runge transitive} in a very strong sense, for any compact set $K\subset \T$ different from $\T$. We refer to \cite[Definition 3.1]{BonGe} for the definition of Runge transitive operators, a definition that can easily be extended to general sequences of operators. Here, if $X$ and $Y$ were Banach spaces, we may say that a sequence $(T_n)_n:X\to Y$ is Runge transitive "in a very strong sense" if it satisfies the following: given any non-empty open set $U\subset X$, any non-empty open sets $V_1,V_2 \subset Y$ and any finite sets $E_1\subset \N$ and $E_2\subset \N$ with $E_1\cap E_2=\emptyset$, there exists $x\in U$ such that for any $n_1\in E_1$ and any $n_2\in E_2$,
		\[
		T_{n_1}(x)\in V_1\quad \text{and}\quad T_{n_2}(x)\in V_2.
		\]
}
\end{rem}

\medskip

In view of the previous remark and \cite[Theorem 3.3]{BonGe}, it is natural to expect the existence of \emph{frequently} Abel universal functions. The rest of this section will confirm this. Let us recall that frequent universality is a generalization of the very important notion of frequent hypercyclicity, introduced by Bayart and Grivaux \cite{BayGri}. Another important notion in linear dynamics is is that of upper frequent universality, introduced by Shkarin \cite{Shka}, which also makes sense in the theory of universality. In order to give definitions, let us introduce some terminology. We define the lower density (resp. the upper density) of a subset $E$ of $\N$, denoted by $\underline{d}(E)$ (resp. $\overline{d}(E)$), as the quantity
\[
\underline{d}(E):=\liminf _{n\to \infty}\frac{|E\cap \{1,\ldots,n\}|}{n}\quad \text{(resp.}\,\limsup _{n\to \infty}\frac{|E\cap \{1,\ldots,n\}|}{n}\text{)},
\]
where the notation $|E|$ denotes the cardinality of $E$. Note that we have $\overline{d}(E):=1-\underline{d}(\N\setminus E)$.
Then a sequence $(T_n)_n$ of continuous linear operators from one Fréchet space $X$ to another Fréchet space $Y$ is said to be frequently universal (resp. upper frequently universal) if there exists $x\in X$ such that for any non-empty open set $U$ of $Y$, the set
\[
\{n\in \N:\,T_n(x)\in U\}
\]
has positive lower density (resp. positive upper density). Such a vector $x$ is said to be frequently hypercyclic (resp. upper frequently hypercyclic) for $(T_n)_n$.

As recalled in the introduction, many concrete examples of frequently universal sequences of iterates $(T^n)_n$ of a single continuous operator $T$ from a Fréchet space to itself have been exhibited (we refer the reader to the books \cite{BayMat,gp}). However, rather few (explicit and natural) examples of frequently universal sequences $(T_n)_n$, which are not iterates of a single operator, are known (see \cite{CEMM} for example). 

\medskip

Thus it is natural to wonder whether Abel universal functions may be (upper) frequently universal. Let us specify what we mean by a \textit{(upper) frequently Abel universal functions}.

\begin{defi}\label{defi-FU-Abel}Let $f\in H(\D)$.
	\begin{enumerate}\item The function $f$ is said to be \emph{$\varrho$-frequently Abel universal} if it satisfies the following property: for every compact set $K\subset \T$ different from $\T$ and any $V\in \VV(K)$, the set $N_{f}(V,K,\varrho)$ has positive lower density.
		
		We will denote by $\mathcal{FU}_A(\mathbb{D},\varrho)$ the set of all $\varrho$-frequently Abel universal functions.
		\item The function $f$ is said to be \emph{$\varrho$-upper frequently Abel universal} if it satisfies the following property: for every compact set $K\subset \T$ different from $\T$ and any $V\in \VV(K)$, the set $N_{f}(V,K,\varrho)$ has positive upper density.
		
		We will denote by $\mathcal{FU}^u_A(\mathbb{D},\varrho)$ the set of all $\varrho$-upper frequently Abel universal functions.
	\end{enumerate}
\end{defi}

The definition of the set $\UU_A(\D)$ invites us to propose a variant of the previous definitions. Indeed, there exist notions of lower and upper density for subsets of $[0,1)$. Let $\Gamma \subset [0,1)$. We call \textit{uniform lower density} of $\Gamma$ the quantity defined by
\[
\underline{d_u}(\Gamma)=\liminf_{n\to \infty}n |\Gamma\cap [1-1/n,1)|,
\]
where $|E|$ denotes de Lebesgue measure of a set $E\subset \T$. Similarly, the \textit{uniform upper density} of $\Gamma$ is defined by
$\overline{d_u}(\Gamma)=1-\underline{d_u}([0,1)\setminus\Gamma)$. It is also easily seen that
\[
\overline{d_u}(\Gamma)=\limsup_{n\rightarrow +\infty} \left(n\vert \Gamma\cap[1-1/n,1)\vert\right).
\]
This leads us to the following definition.

\begin{defi}\label{defi-unif-FU-Abel}Let $f\in H(\D)$.
	\begin{enumerate}\item The function $f$ is said to be \emph{frequently Abel universal} if it satisfies the following property: for every compact set $K\subset \T$ different from $\T$ and any $V\in \VV(K)$, the set $N_{f}(V,K)$ has positive uniform lower density.
		
		We will denote by $\mathcal{FU}_A(\mathbb{D})$ the set of all frequently Abel universal functions.
		
		\item The function $f$ is said to be \emph{upper frequently Abel universal} if it satisfies the following property: for every compact set $K\subset \T$ different from $\T$ and any $V\in \VV(K)$, the set $N_{f}(V,K)$ has positive uniform upper density.
		
		We will denote by $\mathcal{FU}^u_A(\mathbb{D})$ the set of all frequently Abel universal functions.
	\end{enumerate}
\end{defi}

The following lemma gives examples of sets with positive lower or upper densities that are useful for our purpose.

\begin{lemme}\label{lem-sets-density-utiles}$\,$
	\begin{enumerate}
		\item There exist pairwise disjoint subsets $A(l,n)$ of $\N$ such that $\underline{d}(A(l,n))>0$ for any $l,n\in \N$;
		\item There exist pairwise disjoint subsets $A(l,n)$ of $\N$ such that $\overline{d}(A(l,n))=1$ for any $l,n\in \N$;
		\item There exist a locally finite family $\{\Gamma(l,n,m):\,l,n,m\in \N\}$ of pairwise disjoint segments in $[0,1)$ such that $\underline{d_u}(\Gamma(l,n))>0$ for any $l,n\in \N$;
		\item  There exist a locally finite family $\{\Gamma(l,n,m):\,l,n,m\in \N\}$ of pairwise disjoint segments in $[0,1)$ such that $\overline{d_u}(\Gamma(l,n))=1$ for any $l,n\in \N$.
	\end{enumerate}
\end{lemme}

\begin{proof}The assertion (1) is contained in \cite[Lemma 9.5]{gp}. Let us prove (3). According to \cite[Lemma 9.5]{gp}, there exist pairwise disjoint subsets $A(l,n)$, $l,n\geq 1$, of $\N$, with positive lower density, such that for any $j\in A(l,n)$ and $j'\in A(l',n')$, one has
	\begin{equation}\label{eq-lem-unif-density}j\geq n\quad \text{and}\quad |j-j'|\geq n+n'\quad\text{if}\quad j\neq j'.
	\end{equation}
	For any $l,n\geq 1$, let us write
	\[
	A(l,n)=\{j_m^{l,n}\ ,\ m\geq 1\},
	\]
	where $(j_m^{l,n})_m$ is increasing. Now, for $l,n,m\geq 1$, we set
	\[
	\Gamma(l,n,m):=\left[1-\frac{1}{j_m^{l,n}}-\frac{1}{3}\left(\frac{1}{j_m^{l,n}-1}-\frac{1}{j_m^{l,n}}\right),
	1-\frac{1}{j_m^{l,n}}+\frac{1}{3}\left(\frac{1}{j_m^{l,n}}-\frac{1}{j_m^{l,n}+1}\right)
	\right]
	\]
	and for $l,n\geq 1$,
	\[
	\Gamma(l,n):=\bigcup_{m\geq 1}\Gamma(l,n,m).
	\]
	Using \eqref{eq-lem-unif-density}, it is not difficult to check that the sets $\Gamma(l,n)$, $l,n\geq 1$, are pairwise disjoint and that $r\geq 1-\frac{1}{n-1}$ for any $r\in \Gamma(l,n)$. This shows that the family $\{\Gamma(l,n,m):\,l,n,m\in \N\}$ is locally finite. Now, let us fix $l,n\geq 1$ and simply denote $j_m^{l,n}=j_m$, $m\geq 1$. An easy calculation leads to
	\[
	|\Gamma(l,n)|=\frac{2}{3}\sum_{m\geq 1}\frac{1}{j_m^2-1}.
	\]
	Now, taking into account that $A(l,n)$ has positive lower density, there exist $M>1$ such that $m\leq j_m\leq Mm$, $m\geq 1$. Therefore, for any integer $N\geq 1$, we get
	\begin{eqnarray*}
		N\left|\Gamma(l,n)\cap [1-1/N,1)\right| & \geq &
		\frac{2N}{3}\sum_{\substack{m\geq 1\\j_m\geq N}}\frac{1}{j_m^2-1}\\
		&\geq & \frac{2N}{3}\sum_{m\geq N}\frac{1}{M^2m^2}\\
		&\geq & \frac{2N}{3M^2N}=\frac{2}{3M^2}.
	\end{eqnarray*}
	This finishes the proof of (3).
	
	The proof of (2) is standard and similar to that of (4), so we only check (4). By induction, one can build two increasing sequences $(N_i)_i$ and $(M_i)_i$ of positive integers such that $2^{N_i-M_i}\to 0$ as $i\to \infty$ and such that $N_{i+1}>M_i$, $i\in \N$. Then we set $\Gamma_i=[1-2^{-N_i},1-2^{-M_i}]$ and fix increasing sequences $(u_{m}(l,n))_m$, $l,n\in \N$, such that $u_m(l,n)\neq u_{m'}(l',n')$ whenever $(m,l,n)\neq (m',l',n')$. It is now easily checked that the family $\{\Gamma_{u_m(l,n)}:\,l,n,m\in \N\}$ is locally finite and that the sets
	\[
	\Gamma(n,l):=\bigcup_{m\in \N}\Gamma_{u_m(l,n)},\quad l,n\in \N
	\]
	have upper density equal to $1$.
\end{proof}

\begin{rem}{\rm We notice that if $\{\Gamma(l,n,m):\,l,n,m\in \N\}$ is a family of pairwise disjoint segments in $[0,1)$ such that $\underline{d_u}(\Gamma(l,n))>0$ (resp. $\overline{d_u}(\Gamma(l,n))>0$) for any $l,n\in \N$, then every set $\Gamma(l,n)\cap [r,1)$, $l,n\in \N$, $r\in[0,1)$, also has positive uniform lower density (resp. uniform upper density)  (and is of course infinite).}
\end{rem}

We can immediately deduce from Proposition \ref{key-ingredient-FU} and the previous lemma the existence of functions satisfying Definitions \ref{defi-FU-Abel} or \ref{defi-unif-FU-Abel}. But we can also say a bit more. We recall that a subset $A$ of a Fréchet space is first category if it is contained in the complement of a residual set.

\begin{coro}\label{coro-density}$\,$
	\begin{enumerate}
		\item The sets $\mathcal{FU}^u_A(\mathbb{D},\varrho)$ and $\mathcal{FU}^u_A(\mathbb{D})$ are residual in $H(\D)$.
		\item The sets $\mathcal{FU}_A(\mathbb{D},\varrho)$ and $\mathcal{FU}_A(\mathbb{D})$ are dense and first category in $H(\D)$.
	\end{enumerate}
\end{coro}

\begin{proof}Setting
	\[
	\Gamma(l,n)=\{r_k:\,k\in A(l,n)\},
	\]
	where the sets $A(l,n)$, $l,n\in \N$, are given by point (2) of Lemma \ref{lem-sets-density-utiles}, and applying Proposition \ref{key-ingredient-FU}, we get that the set $\mathcal{FU}^{u=1}_A(\mathbb{D},\varrho)$ consisting of those $f\in H(\D)$ such that $\overline{d}(N_f(V,K,\varrho))=1$ for any compact set $K\subset \T$ different from $\T$ and any $V\in \VV(K)$, is dense in $H(\D)$. Similarly, if the sets $A(l,n)$, $l,n\in \N$, are instead given by point (4), then Proposition \ref{key-ingredient-FU} gives us that the set $\mathcal{FU}^{u=1}_A(\mathbb{D})$ consisting of those $f\in H(\D)$ such that $\overline{d_u}(N_f(V,K))=1$ for any compact set $K\subset \T$ different from $\T$ and any $V\in \VV(K)$, is dense in $H(\D)$.
	
	Moreover, in the same way, Proposition \ref{key-ingredient-FU} and points (1) and (3) of Lemma \ref{lem-sets-density-utiles} immediately tell us that the sets $\mathcal{FU}_A(\mathbb{D},\varrho)$ and $\mathcal{FU}_A(\mathbb{D})$ are dense in $H(\D)$. Since $\underline{d}(E)=1-\overline{d}(\N\setminus E)$, $E\in \N$ (resp. $\underline{d_u}(\Gamma)=1-\overline{d_u}([0,1)\setminus \Gamma)$, $\Gamma\subset [0,1)$), the corollary is proved once we have shown that the sets $\mathcal{FU}^{u=1}_A(\mathbb{D},\varrho)$ and $\mathcal{FU}^{u=1}_A(\mathbb{D})$ are $G_{\delta}$ in $H(\D)$. The proof for both sets being very similar, we only deal with $\mathcal{FU}^{u=1}_A(\mathbb{D})$.
	
	To do so, note that
	\[
	\mathcal{FU}^{u=1}_A(\mathbb{D})=\bigcap_{l\in \N}\bigcap_{n\in \N}\bigcap_{m\geq 1}\bigcap_{k\geq 1}\bigcup_{N\geq k}\left\{f\in H(\D):\,\left|N_f(V_n(l),K_l)\right|\geq \left(1-\frac{1}{m}\right)N\right\}.
	\]
	Since each set $\left\{f\in H(\D):\,\left|N_f(V_n(l),K_l)\right|\geq \left(1-\frac{1}{m}\right)N\right\}$ is easily seen to be open, the proof is complete.
\end{proof}

\begin{rem}{\rm The previous proof shows that the sets $\mathcal{FU}^{u=1}_A(\mathbb{D},\varrho)$ and $\mathcal{FU}^{u=1}_A(\mathbb{D})$ are $G_{\delta}$ subsets of $H(\D)$, a conclusion which is stronger than that stated in point (1) of Corollary \ref{coro-density}. We mention that it was shown in \cite{MouMunGlasgow} that \emph{all} universal Taylor series are upper frequently universal, with upper densities of the relevant sets equal to $1$.}
\end{rem}

Since $\mathcal{FU}^{u=1}_A(\mathbb{D},\varrho)$ and $\mathcal{FU}^{u=1}_A(\mathbb{D})$ are residual in $H(\D)$ (see the proof of Corollary \ref{coro-density} for the definition of these sets), it is not difficult to check that
\begin{eqnarray*}
	H(\D) & = & \mathcal{FU}^{u=1}_A(\mathbb{D},\varrho) + \mathcal{FU}^{u=1}_A(\mathbb{D},\varrho)\\
	& = & \mathcal{FU}^u_A(\mathbb{D}) + \mathcal{FU}^u_A(\mathbb{D}).
\end{eqnarray*}
In fact, it is also true that any function in $H(\D)$ can be written as the sum of two functions in  $\mathcal{FU}_A(\mathbb{D})$ or in $\mathcal{FU}_A(\mathbb{D},\varrho)$, even if the latter sets are first category. This is contained in the next proposition.

\begin{prop} We have 
	\begin{eqnarray*}
		H(\D) & = & \mathcal{FU}_A(\mathbb{D},\varrho) + \mathcal{FU}_A(\mathbb{D},\varrho)\\
		& = & \mathcal{FU}_A(\mathbb{D}) + \mathcal{FU}_A(\mathbb{D}).
	\end{eqnarray*}
\end{prop}

\begin{proof} The proofs of the two equalities being analogous to that of \cite[Theorem 5.9]{BonGe}, we only sketch the proof of the first equality. Let $h$ be in $H(\mathbb{D})$ and let $A(l,n)$, $l,n\in \N$, be given by point (1) of Lemma \ref{lem-sets-density-utiles}. We split each $A(l,n)$ into two disjoint sets $A^1(l,n),A^2(l,n)$ with positive lower density. Proceeding as in the proof of Theorem \ref{key-ingredient-FU}, we build $f\in H(\D)$ such that for any $k\in A^1(l,n)$, $f\circ \phi_{r_k}\in V_n(l)$ and for any $k\in A^2(l,n)$, $f\circ \phi_{r_k} \in h\circ \phi^{-1}_{r_k}- V_n(l)$, where $h\circ \phi^{-1}_{r_k}- V_n(l):=\{h\circ \phi^{-1}_{r_k} - g:\,g\in V_n(l)\}$. Then $f$ and $h-f$ belong to $\mathcal{FU}_A(\mathbb{D},\varrho)$ and we obviously have $h=f+(h-f)$.
\end{proof}

\medskip

We shall now continue our illustration of how flexible the notion of Abel universal functions is. It is usually difficult to prove results of \emph{common frequent universality} (see \cite{Bay-com,CEMM}). In the next two results, we are interested in common (frequent) $\varrho(\lambda)$-Abel universality with respect to general sequences $(\rho(\lambda))_{\lambda\in\Lambda}$, $\Lambda\subset \R$. First we state a result for common universality inspired by \cite[Theorem 11.5]{gp}. We shall say that a set is $\sigma$-compact if it is a countable union of compact sets.

\begin{prop}\label{prop1-common-general} Let $\Lambda\subset\mathbb{R}$ be a $\sigma$-compact set and $(\rho(\lambda))_{\lambda\in\Lambda}$ a family of increasing sequences $(r_n(\lambda))_n$. We assume that $r_n(\lambda) \to 1$ for any $\lambda \in \Lambda$ and that the map $\lambda \mapsto r_n(\lambda)$ is continuous on $\Lambda$ for any $n\in \N$. Then
	\[
	\bigcap_{\lambda\in\Lambda}\mathcal{U}_A(\mathbb{D}, \varrho(\lambda))\hbox{ is a }G_\delta\hbox{-dense set in }H(\mathbb{D}).
	\]
\end{prop}

\begin{proof} Let $(I_m)_m$ be a sequence of compact sets whose union is $\Lambda$. Clearly we have
	\[
	\bigcap_{\lambda\in\Lambda}\mathcal{U}_A(\mathbb{D}, \varrho(\lambda))=\bigcap_{j,l,m,s\geq 1}E(j,l,m,s),
	\]
	where 
	\[
	E(j,l,m,s)= \left\{f\in H(\mathbb{D}):\,\forall\lambda\in I_m,\, \exists n\geq 0,\,	\sup_{\zeta\in K_l}\vert f(r_{n}(\lambda)\zeta)-\varphi_j(\zeta)\vert<1/s\right\}.
	\]
	By compactness of $I_m$ and continuity of $\lambda \mapsto r_n(\lambda)$ on $\Lambda$, $n\in \N$, each set $E(j,l,m,s)$, $j,l,m,s\geq 1$, is open. By Baire Category Theorem, it is enough to prove that $E(j,l,m,s)$ is dense in $H(\mathbb{D})$ for any $j,l,m,s\geq 1$. Let then fix $j,l,m,s\geq 1$ and $g\in H(\mathbb{D})$, $0<r<1$ and $\varepsilon >0$. Let us set $r<\tilde{r}<1$. Now observe that the assumptions and Dini's theorem imply that $(r_n(\cdot))_n$ tends to $1$ uniformly on each compact set $I_m$. Therefore there exists $c_{n,I_m}>0$ with $c_{n,I_m}\rightarrow 0$ as $n\rightarrow +\infty$ so that $\vert r_n(\lambda)-1\vert\leq c_{n,I_m}$ for any $\lambda \in I_m$, whence one can find $N_1\geq 1$ such that 
	\[
	\forall\lambda\in I_m,\ \forall n\geq N_1,\ r_n(\lambda)\geq \tilde{r}.
	\]
	By Mergelyan's theorem there exists a polynomial $f$ such that
	\[
	\sup_{\vert z\vert\leq r}\vert f(z)-g(z)\vert<\varepsilon\hbox{ and }
	\sup_{z\in[\tilde{r},1]K_l}\vert f(z)-\varphi_j(z)\vert<1/2s.
	\]
	We get, for any $\lambda\in I_m$ and any $n\geq N_1$,
	\[
	\begin{array}{rcl}\displaystyle\sup_{\zeta\in K_l}\vert f(r_{n}(\lambda)\zeta)-\varphi_j(\zeta)\vert&\leq &\displaystyle
	\sup_{\zeta\in K_l}\vert f(r_{n}(\lambda)\zeta)-\varphi_j(r_{n}(\lambda)\zeta)\vert+\sup_{\zeta\in K_l}\vert \varphi_j(r_{n}(\lambda)\zeta)-\varphi_j(\zeta)\vert\\
	&\leq&\displaystyle
	\frac{1}{2s}+\sup_{\zeta\in K_l}\vert \varphi_j(r_{n}(\lambda)\zeta)-\varphi_j(\zeta)\vert.\end{array}
	\]
	Using again the property $\vert r_n(\lambda)-1\vert\leq c_{n,I_m}$ with $c_{n,I_m}\rightarrow 0$  one can find $N_2\geq N_1$ such that for all $\lambda \in I_m$ and $n\geq N_2$, $\sup_{\zeta\in K_l}\vert \varphi_j(r_{n}(\lambda)\zeta)-\varphi_j(\zeta)\vert <1/2s.$ This completes the proof.
\end{proof}

\medskip

We saw that Proposition \ref{key-ingredient-FU} allows to exhibit $\varrho$-frequently Abel universal functions for any sequence $\varrho=(r_n)_n$ (with $r_n\to 1$ as $n\to \infty$). Actually, it can similarly allow us to prove the existence of functions $f\in H(\D)$ that are $\varrho(l)$-frequently Abel universal simultaneously for countably many sequences $\varrho(l)$, $l\in \N$, up to some assumptions. This is the content of the next proposition, whose proof is left to the reader.

\begin{prop}\label{prop-comFUA}For $l\in \N$, let $\varrho(l)=(r_n^l)_n$ be increasing sequences in $(0,1)$ such that $r_n^l\to 1$ as $n\to \infty$. We assume that for any $l\neq l'$, $\{r_n^l:\,n\in \N\}\cap \{r_n^{l'}:\,n\in \N\}=\emptyset$. Then
\[
\bigcap_{l\in \N}\mathcal{FU}_A(\D,\varrho(l))\neq \emptyset.
\]
\end{prop}

So far we know no \textit{general} result of common $\varrho(\lambda)$-frequent Abel universality, where $\lambda$ belongs to some \textit{uncountable} families of sequences. This seems to us to be an interesting problem.

We shall end the paper by pointing out, however, that frequently Abel universal functions are easily seen to be automatically frequently universal for some uncountable families of composition operators. More precisely let us consider a continuous function $h:(0,1]\times [0,1)\rightarrow \mathbb (0,1)$ such that $h(1,r)=r$, $h(a,r)\rightarrow 1$ as $a\rightarrow 0$ and $h(a,r)\rightarrow 1$ as $r\rightarrow 1$. Assume that for all $a\in (0,1]$, the function $h_a:=h(a,\cdot):[0,1)\rightarrow \mathbb (0,1)$ is increasing, differentiable on $[0,1)$ and
\[
c\leq \liminf _{r\to 1} h_a'(r) \leq \limsup_{r\to 1} h_a'(r) \leq C,
\]
where $c,C$ are positive constants (possibly depending on $a$).
Let us denote by $h_a^{-1}$ the reciprocal function of $h_a$. For $a\in (0,1]$ and $r\in [0,1)$, we define $\phi^a_{h,r}:=\phi_{h(a,r)}$. Let us say that $f\in H(\D)$ is universal (resp. frequently universal, resp. upper frequently universal) for the family $\{\phi^a_{h,r}:\,r\in [0,1)\}$, $a\in (0,1]$ fixed, if for every compact set $K\subset \T$ different from $\T$, and any $V\in \VV(K)$, the set
\[
N^a_f(V,K):=\{r\in [0,1):\,f\circ \phi^a_{h,r}\in V\}
\]
is non-empty (resp. has positive uniform lower density, resp. has positive uniform upper density). Note that, by definition, $f$ is universal (resp. frequently universal, resp. upper frequently universal) for $\{\phi^1_{h,r}:\,r\in [0,1)\}$ if $f$ is in $\UU_A(\D)$ (resp. in $\mathcal{FU}_A(\D)$, resp. in $\mathcal{FU}^u_A(\D)$). 

\begin{prop}\label{Abel_FHC-continuous} Under the previous assumptions, any function $f$ in $\UU_A(\D)$ (resp. any $f$ in $\mathcal{FU}_A(\D)$, resp. any $f$ in $\mathcal{FU}^u_A(\D)$) is simultaneously universal (resp. frequently universal, resp. upper frequently universal) for each family $\{\phi_{h,r}^a:\,r\in [0,1)\}$, $a\in (0,1]$. 
\end{prop}

\begin{proof} Let us fix $f\in H(\D)$, $K\subset \T$ different from $\T$, $V\in \VV(K)$, $a\in (0,1)$, and set $N^1_f(V,K,a):=N^1_f(V,K)\cap (h_a(0),1)$. We note that
	\[
h_a^{-1}\left(N^1_f(V,K,a)\right)\subset N^a_f(V,K),
	\]
and recall that $N^1_f(V,K,a)$ is infinite whenever $N^1_f(V,K)$ is infinite. This proves that if $f\in \UU_A(\D)$ then $f$ is simultaneously universal for each family $\{\phi^a_{h,r}:\,r\in [0,1)\}$, $a\in (0,1]$. Note also that if $N^1_f(V,K)$ has positive lower (resp. upper) density, so has $N^1_f(V,K,a)$. 
To finish the proof, it is thus enough to check that if $N^1_f(V,K,a)$ has positive lower (resp. upper) density, then $h_a^{-1}\left(N^1_f(V,K,a)\right)$ also has positive lower (resp. upper) density. This may be a consequence of a known general result about densities, but we do not have found it in the literature. So let us give the details. This is essentially a change of variable. Indeed, by assumption, there exist two positive constants $0<c\leq C$ such that $c\leq h_a'(r) \leq C$ for any $r\in(0,1)$ large enough. Then the change of variable formula gives for any such $r$,
	\begin{equation}\label{change-var}
\left|h_a^{-1}\left(N^1_f(V,K,a)\right)\cap \left(r,1\right)\right|\geq \frac{1}{C}\displaystyle \left|N^1_f(V,K,a)\cap \left(h_a(r),1\right)\right|.
	\end{equation}
Choosing $r=1-1/N$ (resp. $r=r_N^a:=h_a^{-1}(1-1/N)$) we get from \eqref{change-var} and $N$ large enough,
\begin{equation}\label{eq-low}
	\left|h_a^{-1}\left(N^1_f(V,K,a)\right)\cap \left(1-1/N,1\right)\right|\geq \frac{1}{C}\displaystyle \left|N^1_f(V,K,a)\cap \left(h_a(1-1/N),1\right)\right|
\end{equation}
(resp.
\begin{equation}\label{eq-up}
\left|h_a^{-1}\left(N^1_f(V,K,a)\right)\cap \left(r_N^a,1\right)\right|\geq \frac{1}{C}\displaystyle \left|N^1_f(V,K,a)\cap \left(1-1/N,1\right)\right|.)
\end{equation}
If we set $t_N^a:=(1-h_a(1-1/N))^{-1}$, from \eqref{eq-low} we obtain
\begin{multline*}
N\left|h_a^{-1}\left(N^1_f(V,K,a)\right)\cap \left(1-\displaystyle\frac{1}{N},1\right)\right| \geq \frac{N}{C}\left|N^1_f(V,K,a)\cap \displaystyle\left(1-\frac{1}{\lfloor t_N^a \rfloor+1},1\right)\right|\\
= \displaystyle\frac{N}{C(\lfloor t_N^a \rfloor+1)}\displaystyle \left(\lfloor t_N^a \rfloor+1\right)\left|N^1_f(V,K,a)\cap \left(1-\frac{1}{\lfloor t_N^a \rfloor+1},1\right)\right|.
\end{multline*}
Since by assumptions $\liminf_{N\to \infty}\frac{N}{\lfloor t_N^a \rfloor+1}\geq c>0$ we deduce that
\[
\dlow\left(h_a^{-1}\left(N^1_f(V,K,a)\right)\right)\geq \frac{c}{C}\ \dlow(N^1_f(V,K,a)),
\]
hence that $f$ is frequently universal for $\{\phi_{h,r}^a:\,r\in [0,1)\}$ whenever $f\in \mathcal{FU}_A(\D)$.

Similarly, from \eqref{eq-up} (with $r_N^a=h_a^{-1}(1-1/N)$) we obtain, if we set this time $t_N^a:=(1-r_N^a)^{-1}$,
\begin{equation*}
\lfloor t_N^a\rfloor\left|h_a^{-1}\left(N^1_f(V,K,a)\right)\cap \left(1-1/\lfloor t_N^a\rfloor,1\right)\right|\geq \frac{\lfloor t_N^a\rfloor}{CN}N\displaystyle \left|N^1_f(V,K,a)\cap \left(1-1/N,1\right)\right|
\end{equation*}
Since $\limsup_{N\to \infty}\frac{\lfloor t_N^a \rfloor+1}{N}\geq c>0$ we deduce that
\[
\overline{d}\left(h_a^{-1}\left(N^1_f(V,K,a)\right)\right)\geq \frac{c}{C}\ \overline{d}(N^1_f(V,K,a)),
\]
hence that $f$ is upper frequently universal for $\{\phi_{h,r}^a:\,r\in [0,1)\}$ whenever $f\in \mathcal{FU}^u_A(\D)$.
\end{proof}

Thus Proposition \ref{Abel_FHC-continuous} ensures that the following property holds: there exists $f\in H(\mathbb{D})$ such that for all $a\in (0,1)$, for any compact set  $K\subset\mathbb{T}$ different from $\mathbb{T}$, any continuous function $g$ on $K$ and any $\varepsilon>0$, there exists $\Gamma\subset (0,1)$ with positive uniform lower density such that, for any $r\in \Gamma$, $\sup_K\vert f\circ\phi_{h,r}^a-g \vert<\varepsilon$. Moreover observe that examples of function $h$ satisfying the assumptions of the preceding proposition are given by $(a,r)\mapsto ar+(1-a)$, or $(a,r)\mapsto 1-\frac{a(1-r)}{a+(1-r)(1-a)}$. 

\vskip4mm

\noindent\textbf{Acknowledgment.} We would like to thank the anonymous referee for her/his useful suggestions and comments that improved the quality of the manuscript.

\end{document}